%% file: metric2_journal.tex
\documentclass{amsart}

\usepackage[foot]{amsaddr}

\usepackage{bbm}	
\usepackage{mathrsfs}	

\usepackage{amsmath, amsthm, amssymb}	

\usepackage{enumitem}	

\usepackage{graphicx}

\usepackage{tikz}	
\usetikzlibrary{arrows,shapes,trees,calc}

\usepackage{subcaption}

\usepackage[linktocpage,hidelinks]{hyperref}
\usepackage{mathtools}

\setlist[enumerate]{label=(\roman*)}

\newtheorem{theorem}{Theorem}[section]
\newtheorem{lemma}[theorem]{Lemma}
\newtheorem{corollary}[theorem]{Corollary}

\theoremstyle{definition}

\theoremstyle{remark}
\newtheorem{remark}[theorem]{Remark}

\numberwithin{equation}{section}

\input{0_commands.tex}

\input{0_commands_wt.tex}

\begin{document}

\title[Brownian Motions on Metric Graphs II: Construction]{Brownian Motions on Metric Graphs with Non-Local Boundary Conditions II: Construction}

\author[Florian Werner]{Florian Werner}
\address{Institut f\"ur Mathematik, Universit\"at Mannheim, 68131 Mannheim, Germany}
\email{fwerner@math.uni-mannheim.de}



\date{\today}



\newcommand{\sqed}{\relax}

\begin{abstract}
  \input{metric2_abstract.tex}
\end{abstract}

\maketitle

\input{metric2_intro.tex}

\input{metric2_content.tex}

\section*{Acknowledgements}
\input{metric2_acknowledgements.tex}

\bibliographystyle{amsplain}
\bibliography{dissMG}

\end{document}

%% file: 0_commands.tex
\newcommand{\N}{\mathbb{N}}	

\newcommand{\sB}{\mathscr{B}}	
\newcommand{\sE}{\mathscr{E}}	

\newcommand{\cC}{\mathcal{C}}	

\newcommand{\BB}{B}		

\renewcommand{\O}{\Omega}	
\renewcommand{\o}{\omega}	
\newcommand{\PV}{\mathbb{P}}	
\newcommand{\EV}{\mathbb{E}}	

\newcommand{\sF}{\mathscr{F}}	
\newcommand{\D}{\Delta}		
\newcommand{\sq}{\square}       
\newcommand{\z}{\zeta}		

\renewcommand{\a}{\alpha}	

\renewcommand{\t}{\tau}		
\newcommand{\h}{H}		

\newcommand{\sD}{\mathscr{D}}	

\newcommand{\cG}{\mathcal{G}}	
\newcommand{\cV}{\mathcal{V}}	
\newcommand{\cI}{\mathcal{I}}	
\newcommand{\cE}{\mathcal{E}}	
\newcommand{\cL}{\mathcal{L}}	
\newcommand{\cLV}{\partial}	
\newcommand{\cR}{\rho}		

\newcommand{\tcG}{\widetilde{\cG}}

\newcommand{\tcL}{\widetilde{\cL}}

\newcommand{\bs}{\backslash}	
\newcommand{\comp}{\complement}	

\newcommand{\e}{\varepsilon}	
\renewcommand{\d}{\delta}	        
\newcommand{\p}{\pi}		

\newcommand{\tX}{\widetilde{X}{}} 
\newcommand{\ttX}{\widetilde{\tX}{}} 
\newcommand{\ttcG}{\widetilde{\tcG}{}} 
\newcommand{\ttcL}{\skew{3.5}\wt{\tcL}} 

\renewcommand{\ttcG}{\skew{3.5}\wt{\tcG}} 

\renewcommand{\ttX}{\skew{3.5}\wt{\tX}}
 
\renewcommand{\tt}{\widetilde{\t}{}}
\newcommand{\ttt}{\skew{1.5}\wt{\tt}} 

\newcommand{\tH}{\widetilde{\h}{}}
\newcommand{\ttH}{\skew{1.5}\wt{\tH}}

\newcommand{\tnu}{\widetilde{\nu}{}}
\newcommand{\ttnu}{\skew{1.5}\wt{\tnu}{}} 

\newcommand{\tK}{\widetilde{K}{}}
\newcommand{\ttK}{\skew{1.5}\wt{\tK}{}}

\newcommand{\tmu}{\widetilde{\mu}{}}
\newcommand{\ttmu}{\skew{1.5}\wt{\tmu}{}} 

\newcommand{\tomu}{\widetilde{\overbar{\mu}}{}}
\newcommand{\ttomu}{\skew{1.5}\wt{\tomu}{}} 

\newcommand{\tE}{\widetilde{E}}

\newcommand{\1}{\mathbbm{1}}		
\newcommand{\id}{\operatorname{id}} 	

\newcommand{\abs}[1]{\left\lvert #1 \right\rvert}  

\newcommand\restr[2]{{		  	
  \left.\kern-\nulldelimiterspace 
  #1 				  
  \vphantom{\big|}		  
  \right|_{#2}    		  
  }}

\newcommand{\overbar}[1]{\mkern 1.5mu\overline{\mkern-1.5mu#1\mkern-1.5mu}\mkern 1.5mu}

\newcommand{\supp}{\operatorname{supp}}

\newcommand{\KPS}{Kostrykin, Potthoff and Schrader}

\newcommand{\FW}{Feller\textendash{}Wentzell}



%% file: metric2_abstract.tex
A pathwise construction of discontinuous Brownian motions on metric graphs
is given for every possible set of non-local \FW\ boundary conditions.
This construction is achieved by locally decomposing the metric graphs into star graphs,
establishing local solutions on these partial graphs,
pasting the solutions together, introducing non-local jumps,
and verifying the generator of the resulting process.

%% file: metric2_intro.tex
\section{Introduction}

This article is the final part in a series of works in which we achieve a classification and pathwise construction of Brownian motions on metric graphs.
In~\cite{WernerMetricA}, we defined Brownian motions on metric graphs in accordance with previous works of It\^o and McKean~\cite{ItoMcKean63} 
and \KPS~\cite{KPS12}, that is, as right continuous, strong Markov processes which behave on every edge of the graph like
the standard one-dimensional Brownian motion.
There, we showed that the generator $A = \frac{1}{2} \, \Delta$ of every Brownian motion on a metric graph~$\cG$ satisfies at each vertex point~$v \in \cV$ 
a non-local \FW\ boundary condition
 \begin{align*}
  \forall f \in \sD(A): \quad
  p^v_1 f(v) - \sum_{l \in \cL(v)} p^{v,l}_2 \, f_l'(v) + \frac{p^v_3}{2} f''(v) - \int_{\cG \bs \{v\}} \hspace*{-8.2pt}  \big( f(g) - f(v) \big) \, p^v_4 (dg) = 0
 \end{align*}
for some constants $p^v_1 \geq 0$, $p^{v,l}_2 \geq 0$ for each edge~$l \in \cL(v)$ emanating from~$v$, $p^v_3 \geq 0$ and a measure $p^v_4$ on $\cG \bs \{v\}$,
normalized by
  \begin{align*}
    p^v_1 + \sum_{l \in \cL(v)} p^{v,l}_2 + p^v_3 + \int_{\cG \bs \{v\}} \big( 1 - e^{-d(v,g)} \big) \, p^v_4 (dg) = 1,
  \end{align*}  
and $p^v_4$ being an infinite measure if $\sum_{l \in \cL(v)} p^{v,l}_2 + p^v_3 = 0$.

After having developed the necessary process transformations of concatenations and process revivals in~\cite{WernerConcat},
collected the characteristic properties of Brownian motions on metric graphs in~\cite{WernerMetricA} and
constructed all Brownian motions on star graphs in~\cite{WernerStar},
we are now in the position to give a complete pathwise construction of Brownian motions on any metric graph
for every admissible set of \FW\ data, thus proving the following existence theorem:

\begin{theorem} \label{theo:G_GC:complete construction}
 Let $\cG = (\cV, \cE, \cI, \cLV, \cR)$ be a metric graph,\footnote{As in the previous works, we will assume any metric graph discussed here to have no loops (see~\cite[Section~A.2, Remark~3.1]{WernerMetricA}).
Furthermore, we restrict our attention to metric graphs with finite sets of edges and vertices. A short introduction to metric graphs can be found in~\cite[Appendix~A]{WernerMetricA}.}
 and for every $v \in \cV$ let constants
 $p^v_1 \geq 0$, $p^{v,l}_2 \geq 0$ for each $l \in \cL(v)$, $p^v_3 \geq 0$ and a measure $p^v_4$ on $\cG \bs \{v\}$ be given with
  \begin{align*}
    p^v_1 + \sum_{l \in \cL(v)} p^{v,l}_2 + p^v_3 + \int_{\cG \bs \{v\}} \big( 1 - e^{-d(v,g)} \big) \, p^v_4 (dg) = 1,
  \end{align*}  
  and
  \begin{align*}
    p^v_4 \big( \cG \bs \{v\} \big) = +\infty, \quad \text{if} \quad \sum_{l \in \cL(v)} p^{v,l}_2 + p^v_3 = 0.
  \end{align*}
  Then there exists a Brownian motion $X$ on $\cG$ which is continuous inside all edges, such that its generator satisfies
   \begin{align*}
     \sD(A) \subseteq
       \Big\{ & f \in \cC^2_0(\cG): \forall v \in \cV: \\
          & p^v_1 f(v) - \sum_{l \in \cL(v)} p^{v,l}_2 \, f_l'(v) + \frac{p^v_3}{2} f''(v) - \int_{\cG \bs \{v\}} \big( f(g) - f(v) \big) \, p^v_4 (dg) = 0  \Big\}.
   \end{align*}
\end{theorem}

%% file: metric2_content.tex
\section{Construction Approach} \label{sec:G_GC:agenda}
The construction proceeds as follows: We will begin with Brownian motions on star graphs which implement the
corresponding ``local'' boundary conditions (including ``small jumps'') at their respective vertices. When the process is started on one of these star graphs
and approaches (or jumps to) the vicinity of another vertex, it is killed and revived on the relevant subgraph with the help of concatenation techniques.
That way, we obtain a Brownian motion on a general metric graph by 
successive pastings of partial Brownian motions on star graphs.
The accurate construction approach will be laid out in the following.

 Technically, we will not start with star graphs, but with the complete metric graph which we then decompose into subgraphs. 
 This approach is necessary, as the subgraphs (that is, at some level, star graphs) must be chosen appropriately in order to
 construct the correct complete graph at the end, and the topology of the full graph is required for the pathwise construction and 
 the specification of the \FW\ data.
 
 Let $\cG = (\cV, \cI, \cE, \cLV, \cR)$ be a metric graph having at least two vertices. 
 We will break $\cG$ up by decomposing the set of vertices into  $\cV = \cV^{-1} \uplus \cV^{+1}$ and 
 defining two ``subgraphs'' $\ttcG^{j}$, $j \in \{-1, +1\}$, which possess the respective vertices $\cV^j$ as well as all 
 of the original edges (with their combinatorial structure) not incident with the other vertices~$\cV^{-j}$.
 As internal edges $i$ which are incident with vertices of both subgraphs are lost, we need to replace them by new external ``shadow'' edges 
 $e^{-1}_i$,~$e^{+1}_i$ on the respective subgraphs, see the upper graph of figure~\ref{fig:G_GC:glueing}.
 
 \begin{figure}[tb]
  \centering
  \input{G_GC_glueing_1.tex}
  
  \vspace*{1em}
  
  \input{G_GC_glueing_2.tex}
  \caption[Decomposition and gluing of metric graphs]
          {Decomposition and gluing of metric graphs:
           The metric graph $\cG$ of \cite[Figure~1]{WernerMetricA} is decomposed into two ``subgraphs'' $\ttcG^{-1}$ and $\ttcG^{+1}$
           with vertices $\cV^{-1} = \{v_1, v_2, v_3\}$ and $\cV^{+1} = \{v_4, v_5, v_6\}$.
           The internal edges $i$ which are incident with vertices of both subgraphs are replaced by new external edges 
           $e^{-1}_i$, $e^{+1}_i$ on the respective subgraphs. 
           By performing the transformations explained in section~\ref{sec:G_GC:agenda},
           subsets of these ``subgraphs'' are mapped to the subsets~$\cG^{-1}$, $\cG^{+1}$ of the graph~$\cG$.}  \label{fig:G_GC:glueing}
 \end{figure}
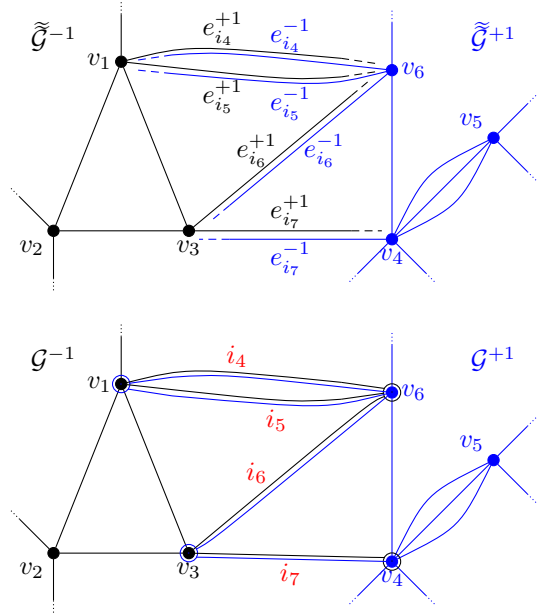 
 
 By iteratively decomposing the subgraphs further up to the level of star graphs,
 we are able to apply our results of~\cite{WernerStar} and introduce Brownian motions on $\ttcG^{-1}$ and~$\ttcG^{+1}$ with the desired boundary behavior at their vertices. 
 In order to paste the two processes---and thus the two graphs---together, we need to cut out the excrescent parts of the external ``shadow'' edges 
 by removing them from the subgraphs and killing the partial Brownian motions whenever they hit the removed locations.
 The remaining parts of these external edges need to be reorientated where necessary (as vertices are always initial points of external edges)
 and then are mapped to the original internal edges in order to get proper subgraphs $\cG^{-1}$ and~$\cG^{+1}$ of the original graph~$\cG$,
 see the lower graph of figure~\ref{fig:G_GC:glueing}.
 
 The resulting Brownian motions on $\cG^{-1}$ and $\cG^{+1}$ can now be pasted together with the help of the alternating copies technique established 
 in~\cite[Section~3]{WernerConcat}, namely by reviving the subprocesses at the other subgraph whenever they leave
 the remaining part of one of their shadow vertices (and thus are killed).

 This construction approach will cause two main technical difficulties, which will prescribe the order of applied transformations:
 Firstly, the ``global'' jumps, that is jumps to other vertices or subgraphs, can only be implemented once the gluing is complete,
    as their jump destinations do not exist for the original Brownian motions on the subgraphs.
    They will be implemented by an instant return process with an appropriate revival measure.
 Moreover, the implementation of the killing portions $(p^v_1, v \in \cV)$ via jumps to the cemetery must be postponed until the gluing procedure and the introduction of 
    the global jumps is complete. The reason is that, as just mentioned, both procedures will apply the technique of identical/alternating copies, which is based on reviving the process
    and would therefore cancel any killing effect beforehand. 
 
 The above-mentioned restrictions and interactions of these techniques lead to some rather unwieldy ``workarounds'' in the upcoming complete construction.
 We are giving an overview of the construction steps now, 
 the mathematical justifications will follow in sections~\ref{sec:G_GC:killing on absorbing set}--\ref{sec:G_GC:completing}.
 
  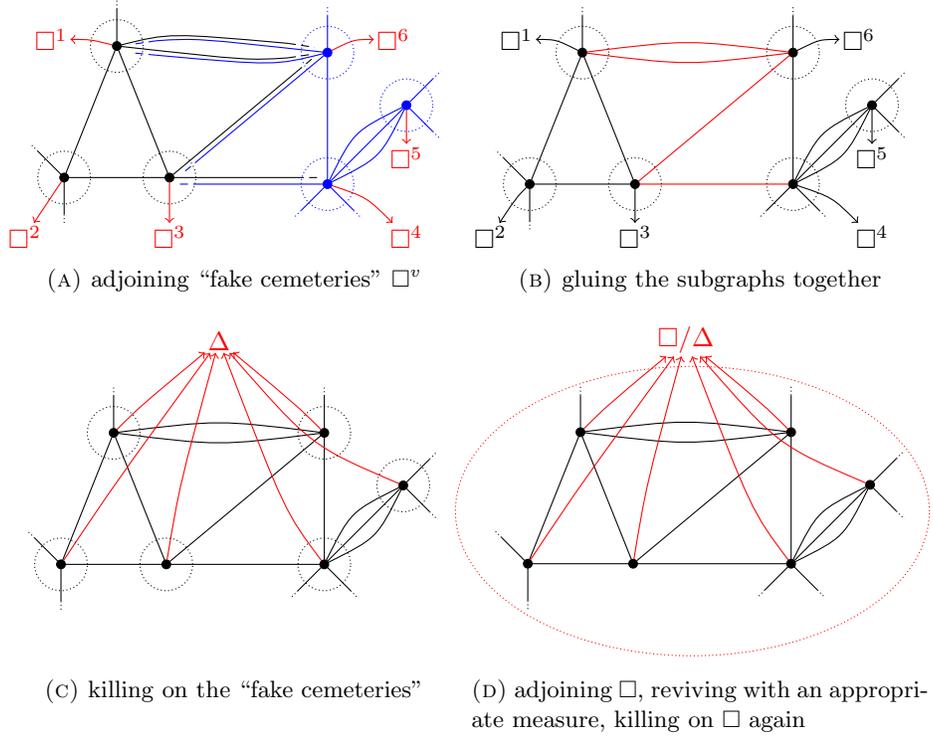
\begin{figure}[tb]
  \centering
    \begin{subfigure}[b]{0.48\textwidth}
     \input{G_GC_constr_1.tex}
     \caption{adjoining ``fake cemeteries'' $\sq^v$} \label{fig:G_GC:construction a}
    \end{subfigure}
    \begin{subfigure}[b]{0.48\textwidth}
     \input{G_GC_constr_2.tex}
     \caption{gluing the subgraphs together} \label{fig:G_GC:construction b}
    \end{subfigure}
    
    \vspace*{1em}
    
    \begin{subfigure}[t]{0.48\textwidth}
     \hspace*{-1.05em}
     \input{G_GC_constr_3.tex}
     \vspace*{-1em}
     \caption{killing on the ``fake cemeteries''} \label{fig:G_GC:construction c}
    \end{subfigure}
    \begin{subfigure}[t]{0.48\textwidth}
     \hspace*{-0.85em}
     \input{G_GC_constr_4.tex}
     \vspace*{-1em}
     \caption{adjoining $\sq$, reviving with an appropriate measure, killing on $\sq$ again} \label{fig:G_GC:construction d}
    \end{subfigure}    
%
%
  \caption[Completing the construction of Brownian motions on a metric graph]
          {Completing the construction of Brownian motions on a metric graph:
           Illustrated are the steps that are
           performed in the construction of the target Brownian motion on the complete graph, 
           when starting with Brownian motions on the subgraphs which already implement the correct reflection, stickiness and ``local'' jump parameters.
           The dotted lines indicate the range of the implemented jump measures.}  \label{fig:G_GC:construction}
 \end{figure}

Assume that we are given a metric graph $\cG = (\cV, \cI, \cE, \cLV, \cR)$ and boundary weights
 \begin{align*}
  \big( p^v_1, (p^{v,l}_2)_{l \in \cL(v)}, p^v_3, p^v_4 \big)_{v \in \cV}
 \end{align*}
which satisfy the conditions of Feller's theorem~\cite[Theorem~1.1]{WernerMetricA}.

As we cannot introduce the distant jumps yet, we choose
for each $v \in \cV$ a distance $\d^v > 0$ such that $\d^v$ is smaller than the lengths of all edges emanating from~$v$, and
define the restricted jump measure 
 \begin{align*}
  q^v_4 & := p^v_4 \big( \, \cdot \, \cap \BB_{\d^v}(v) \big)
 \end{align*}
on the ball~$\BB_{\d^v}(v)$ around~$v$ with radius $\d^v$, and the ``extended'' killing parameter
  \begin{align*}
  q^v_1 & := p^v_1 + p^v_4 \big( \comp \BB_{\d^v}(v) \big). 
 \end{align*}
 
We are going to construct the complete Brownian motion with the just given boundary weights iteratively. That is, we decompose 
the metric graph into two subgraphs $\ttcG^{-1}$ and~$\ttcG^{+1}$ as explained above, and assume that there exist two Brownian motions
$\ttX^{-1}, \ttX^{+1}$ thereon which implement the boundary conditions
 \begin{align*} 
  \big( q^v_1, (p^{v,l}_2)_{l \in \ttcL^j(v)}, p^v_3, q^v_4 \big)_{v \in \cV^j}, \quad j \in \{-1,+1\},
 \end{align*}
where we set the reflection parameters for the adjoined edges to $\smash{p_2^{v,e^{j}_i} = p_2^{v,i}}$.

As the gluing procedure only works for processes with no additional killing effects at the vertices,
we further adjoin for every vertex $v \in \cV$ an absorbing ``fake'' cemetery point $\sq^v$ to the
respective subgraph $\ttcG^{j}$, and assimilate the killing parameter into the jump measure by reviving the subprocesses at $\sq^v$ 
whenever they die at $v$, see figure~\textsc{\ref{fig:G_GC:construction a}}.
Then the new processes possess the boundary conditions
 \begin{align*}
  \big( 0, (p^{v,l}_2)_{l \in \ttcL^j(v)}, p^v_3, q^v_1 \, \e_{\sq^v} + q^v_4 \big)_{v \in \cV^j}, \quad j \in \{-1,+1\}.
 \end{align*}

Next, we glue both processes together and obtain a process on the complete graph $\cG$, as illustrated in figure~\textsc{\ref{fig:G_GC:construction b}}, with boundary conditions 
 \begin{align*}
  \big( 0, (p^{v,l}_2)_{l \in \cL(v)}, p^v_3, q^v_1 \, \e_{\sq^v} + q^v_4 \big)_{v \in \cV}.
 \end{align*}
 
In order to introduce the global jumps, we split the jump to $\sq^v$, with original weight
$q^v_1 = p^v_1 + p^v_4 \big( \comp \BB_{\d^v}(v) \big)$, into killing with weight $p^v_1$
and non-local jumps relative to the measure $p^v_4 \big( \, \cdot \, \cap \comp \BB_{\d^v}(v) \big)$.
To this end, we need to kill the process again: By mapping the absorbing points $\{ \sq^v, v \in \cV \}$ to the ``real'' cemetery $\D$,
see figure~\textsc{\ref{fig:G_GC:construction c}}, we obtain a newly killed process with boundary conditions
 \begin{align*}
  \big( q^v_1, (p^{v,l}_2)_{l \in \cL(v)}, p^v_3, q^v_4 \big)_{v \in \cV}.
 \end{align*}
We adjoin another absorbing ``fake'' cemetery point $\sq$ and construct the next process as instant revival process with revival distribution
 $\big( p^v_1 \, \e_{\sq} + p^v_4 \big( \, \cdot \, \cap \comp \BB_{\d^v}(v) \big) \big) / q^v_1$.
This process now implements jumps relative to the measure 
 $p^v_1 \, \e_{\sq^v} + p^v_4 \big( \, \cdot \, \cap \comp \BB_{\d^v}(v) \big)$, 
which adds to the already existing jump measure $q^v_4 = p^v_4 \big( \, \cdot \, \cap \BB_{\d^v}(v) \big)$.
Thus, this process satisfies the boundary conditions
 \begin{align*}
  \big( 0, (p^{v,l}_2)_{l \in \cL(v)}, p^v_3, p^v_1 \, \e_{\sq} + p^v_4 \big)_{v \in \cV}.
 \end{align*} 
 
Finally, we transform the jumps to $\sq$ into killing by mapping $\sq$ to $\D$, and obtain the complete boundary condition 
 \begin{align*}
  \big( p^v_1, (p^{v,l}_2)_{l \in \cL(v)}, p^v_3, p^v_4 \big)_{v \in \cV}.
 \end{align*} 
 
As seen above, we need to perform many process transformations in the complete construction,
while keeping track of the resulting boundary conditions.
In order to keep our results comprehensible, we first analyze the two main components---killing
 on an absorbing set and introduction of jumps via the instant revival process---together with
their effects on the generator separately in the next two sections.

\section{Killing a Brownian Motion on an Absorbing Set} \label{sec:G_GC:killing on absorbing set}
 
In this section, we examine how killing a Brownian motion on an absorbing set~$F$
affects the boundary conditions of its generator. It will turn out that the jump portion which originally led to $F$
is just transformed into the killing portion, as any jump to $F$ is now immediately triggering the killing.
 
We implement the killing transformation by mapping the absorbing set $F$ to $\D$,
that is, we consider the process $\psi(X)$ for the map
 \begin{align} \label{eq:G_GC:killing mapping}
  \psi \colon \cG \rightarrow \cG \bs F, ~ x \mapsto \psi(x) :=
    \begin{cases}
     x,  & x \in \cG \bs F, \\
     \D, & x \in F.
    \end{cases}
 \end{align}
It has been shown in Appendix~\ref{app:C_MS:killing on absorbing set} that the transformed process $\psi(X)$ is a right process if $X$ is a right process and $F$ is an isolated and absorbing set for $X$.

We are able to obtain the following set of necessary boundary conditions by directly computing the generator of the transformed process:

\begin{lemma} \label{lem:G_GC:killing on absorbing set, generator data}
 Let $X$ be a Brownian motion on $\cG$ with generator 
  \begin{align*}
   \sD(A^X)
   \subseteq \Big\{ & f \in \cC^2_0(\cG): \forall v \in \cV: \\
       &  c^v_1 f(v) - \sum_{l \in \cL(v)} c^{v,l}_2 f_l'(v) + \frac{c_3}{2} f''(v) 
         - \int_{\cG \bs \{v\}} \big( f(g) - f(v) \big) \, c^v_4(dg) = 0 \Big\},
  \end{align*}
 and $F \subsetneq \cG$ be an isolated, absorbing set for~$X$.
 Let $Y := \psi(X)$ be the process on~$\cG \bs F$ resulting from killing~$X$ on~$F$, 
 with~$\psi$ as given in equation~\eqref{eq:G_GC:killing mapping}.
 Then the domain of the generator of~$Y$ satisfies
  \begin{align*}
   & \sD(A^Y) 
   \subseteq  \Big\{  f \in \cC^2_0(\cG \bs F): \forall v \in \cV \bs F: \\
        & \quad \big( c^v_1 + c^v_4(F) \big) f(v) - \sum_{l \in \cL(v)} c^{v,l}_2 f_l'(v) + \frac{c_3}{2} f''(v) 
         - \int_{\cG \bs (F \cup \{v\})} \hspace*{-25pt} \big( f(g) - f(v) \big) \, c^v_4(dg) = 0 \Big\}.
  \end{align*} 
\end{lemma}
\begin{proof}
 For all $f \in \sD(A^Y)$,
 we have for $g \in \cG \bs F$
  \begin{align*}
   A^X (f \circ \psi) (g)
   & = \lim_{t \downarrow 0} \frac{\EV_g \big( f \circ \psi(X_t) \big) - f \circ \psi(g) }{t} \\
   & = \lim_{t \downarrow 0} \frac{\EV_g \big( f(Y_t) \big) - f(g)}{t},
  \end{align*}
 which exists and is equal to $A^Y f(g)$. 
 On the other hand, if $g \in F$, then $X_t \in F$ holds for all $t \geq 0$, $\PV_g$-a.s., because $F$ is absorbing for $X$, and it follows that
  \begin{align*}
   A^X (f \circ \psi) (g)
   & = \lim_{t \downarrow 0} \frac{\EV_g \big( f \circ \psi(X_t) \big) - f \circ \psi(g) }{t}
     = \lim_{t \downarrow 0} \frac{\EV_g \big( f (\D) \big) - f (\D) }{t}
     = 0.
  \end{align*}
 Thus, we have $f \circ \psi \in \sD(A^X)$ for all $f \in \sD(A^Y)$, and $A^X (f \circ \psi) = A^Y f \, \1_{\comp F}$ in this case.
 
 So, if $f \in \sD(A^Y)$, then $f \circ \psi$ fulfills the boundary condition for $X$, that is
  \begin{align*}
    0 & = 
      c^v_1 f \big( \psi (v) \big)- \sum_{l \in \cL(v)} c^{v,l}_2 f_l' \big( \psi (v) \big) + \frac{c_3}{2} f'' \big( \psi (v) \big) \\
     & \qquad - \int_{\cG \bs \{v\}} \big( f \big( \psi (g) \big) - f \big( \psi (v) \big) \big) \, c^v_4(dg) \\
     & = c^v_1 f(v) - \sum_{l \in \cL(v)} c^{v,l}_2 f_l'(v) + \frac{c_3}{2} f''(v) \\
     & \qquad - \int_{\cG \bs (F \cup \{v\})} \big( f(g) - f(v) \big) \, c^v_4(dg) + f(v) \, c^v_4(F)
  \end{align*}
 for all $v \in \cV \bs F$, where we used $f \big( \psi (g) \big) = f(\D) = 0$ for all $g \in F$.
\end{proof}


In general, this proof does not provide us with the \FW\ data of the killed process,
as we are only able to directly compare the \FW\ data with the boundary data of the generator in the star graph case (cf.~\cite[Lemma~4.1]{WernerMetricA}).
Therefore, we need to derive it manually by checking its definitions given in Feller's theorem~\cite[Theorem~1.2]{WernerMetricA}:

\begin{lemma} \label{lem:G_GC:killing on absorbing set, Feller data}
 Let $X$ be a Brownian motion on $\cG$ with \FW\ data  
   \begin{align*}
    \big( c^{v,\D}_1, c^{v,\infty}_1, (c^{v,l}_2)_{l \in \cL(v)}, c^v_3, c^v_4 \big)_{v \in \cV},
   \end{align*}
 and $F \subsetneq \cG$ be an isolated, absorbing set for~$X$.
 Let $Y := \psi(X)$ be the process on~$\cG \bs F$ resulting from killing~$X$ on~$F$, 
 with~$\psi$ as given in equation~\eqref{eq:G_GC:killing mapping}.
 If~$\cG \bs F$ is a metric graph and
 $Y$ is a Brownian motion on~$\cG \bs F$, then the \FW\ data of~$Y$ reads
   \begin{align*} 
    \big( c^{v,\D}_1 + c^v_4(F), c^{v,\infty}_1, (c^{v,l}_2)_{l \in \cL(v)}, c^v_3, c^v_4( \,\cdot\, \cap F^\comp) \big)_{v \in \cV \bs F}.
   \end{align*}
\end{lemma}
\begin{proof}
 We are using the notations of~\cite[Theorem~1.2]{WernerMetricA}, and indicate the corresponding process in the superscript of the variables.
 Fix $v \in \cV \bs F$. 
 The processes' exit behaviors totally coincide, except if~$X$ exits from a small neighborhood of~$v$ by jumping into~$F$ (then $Y$ jumps to~$\D$).
 Thus, $\EV_v(\t^X_\e) = \EV_v(\t^Y_\e)$ holds for all sufficiently small $\e > 0$, and the exit distributions read
  \begin{align*}
   \PV_v \big( Y_{\t^Y_\e} \in dg \cap (\cG \bs F) \big) 
   & = \PV_v \big( X_{\t^X_\e} \in dg \cap (\cG \bs F) \big) , \\
   \PV_v \big( Y_{\t^Y_\e} = \D \big)  
   & = \PV_v\big ( X_{\t^X_\e} \in \{\D\} \cup F \big).
  \end{align*}
 Therefore, we have $\nu^{Y, v}_\e = \nu^{X, v}_\e \big(\,\cdot\, \cap (\cG \bs F) \big)$ and, as $d(v,f) = +\infty$ for all $f \in F$, 
  \begin{align*}
   \int_F \big( 1 - e^{-d(v,g)} \big) \, \nu^{X,v}_\e(dg) = \nu^{X,v}_\e(F) = \frac{\PV_v (X_{\t^X_\e} \in F)}{\EV_v(\t^X_\e)}.
  \end{align*}
 It follows that
  \begin{align*}
   K^{Y, v}_\e
   & = 1 + \frac{\PV_v (Y_{\t^Y_\e} = \D)}{\EV_v(\t^Y_\e)} + \int_{\cG \bs \{v\}} \big( 1 - e^{-d(v,g)} \big) \, \nu^{Y,v}_\e(dg) \\
   & = 1 + \frac{\PV_v (X_{\t^X_\e} = \D)}{\EV_v(\t^X_\e)} + \int_{(\cG \bs \{v\}) \cup F} \big( 1 - e^{-d(v,g)} \big) \, \nu^{X,v}_\e(dg) \\   
   & = K^{X,v}_\e.
  \end{align*}
 As $F$ is isolated, we get $\overbar{\mu}^{Y,v} = \overbar{\mu}^{X,v} \big( \,\cdot\, \cap \big( \overbar{\cG \bs \{v\}} \bs F \big) \big)$, and conclude that
  \begin{align*}
   c_1^{Y,v,\D}
   & = \lim_{\e \downarrow 0} \Big(  \frac{\PV_v (X_{\t^X_\e} = \D)}{\EV_v(\t^X_\e) \, K^{X,v}_\e}
           + \frac{\PV_v (X_{\t^X_\e} \in F)}{\EV_v(\t^X_\e) \, K^{X,v}_\e} \Big) \\
   & = c_1^{X,v,\D} + \overbar{\mu}^{X,v}(F) \\
   & = c_1^{X,v,\D} + c_4^{X,v}(F),
  \end{align*}   
  as well as
   $c_1^{Y,v,\infty} = c_1^{X,v,\infty}$,
   $c_2^{Y,v,l} = c_2^{X,v,l}$ for each $l \in \cL(v)$,
   $c_3^{Y,v} = c_3^{X,v}$,
   and $c_4^{Y,v}  = c_4^{X,v} \big(\,\cdot\, \cap (\cG \bs F) \big)$.
\end{proof}

\begin{remark} \label{rem:G_GC:on killing on absorbing set, Feller data}
 We will apply Lemma~\ref{lem:G_GC:killing on absorbing set, Feller data} in the following context:
 Let $X$ be a Brownian motion on $\cG$ and $F$ be an isolated and absorbing set for $X$, such that for its first entry time 
 $\h_F := \inf \{ t \geq 0: X_t \in F \}$ and $\h_X$ as given in \cite[Definition~2.1]{WernerMetricA}, 
  \begin{align*}
   \h_X < \h_F  \quad \text{$\PV_g$-a.s.}
  \end{align*}
 holds true for all $g \in \comp F$.
 
 It then follows from Theorem~\ref{theo:C_MS:killing on absorbing set} that the killed process $Y = \psi(X)$ is a right process, 
 and therefore strongly Markovian.
 If $\cG \bs F$ is a metric graph, then, as $H_Y = H_X$ and $Y_t = X_t$ for all $t \leq \h_X < \h_F$,
 the properties of \cite[Theorem~2.5]{WernerMetricA} follow for $Y$ from the respective ones of $X$.
 Thus, $Y$ is a Brownian motion on $\cG \bs F$, and Lemma~\ref{lem:G_GC:killing on absorbing set, Feller data} can be applied
 in order to deduce the \FW\ data of $Y$.
 
 In particular, the condition above is satisfied if $F$ can only be reached from $\comp F$ via jumps from vertices, which, as $F$ is isolated and thus
 has positive distance from any vertex $v \in \cV \bs F$, cannot happen immediately due to the normality of the process.
\end{remark}

\section{Introduction of Non-Local Jumps} \label{sec:G_GC:intoduction non-local jumps}

We will introduce the ``global'' jumps, namely jumps to other subgraphs, with the help of the technique of instant revivals
as established in~\cite[Theorem~1.7]{WernerConcat}. 
In order to prepare this approach, we examine the effect of this method on the \FW\ data. 
Similar results were already attained in the examinations concerning Brownian motions on star graphs (see~\cite[Lemma~4.2, Lemma~4.3]{WernerMetricA}).
The next lemma shows that, as expected, the killing weight will be transformed to an additional jump portion with distribution given by the revival kernel.
It also clarifies that this technique can only be used for the implementation of finite jump measures.

\begin{lemma} \label{lem:G_GC:reviving BB with revival kernel}
 Let $X$ be a Brownian motion on $\cG$ with \FW\ data
   \begin{align*}
    \big( c^{v,\D}_1, c^{v,\infty}_1, (c^{v,l}_2)_{l \in \cL(v)}, c^v_3, c^v_4 \big)_{v \in \cV},
   \end{align*}
 lifetime $\z^X$, and exit times $\t^X_\e := \inf \big\{ t \geq 0: d(X_t, X_0) > \e \big\}$ for $\e > 0$.
 If $c^{v,\D}_1 > 0$,
 consider the instant revival process $Y$, constructed from $X$ with the revival kernel
  \begin{align*}
    k(v, \,\cdot\,) = \kappa^v, \quad v \in \cV,
  \end{align*} 
 for some probability measure $\kappa^v$ on $\cG$, and $k(g, \,\cdot\,) = \e_g$ for all $g \notin \cV$. 
 Suppose that for every $v \in \cV$ there exists $\d > 0$ such that
  \begin{enumerate}
   \item $\kappa^v \big( \BB_\d(v) \big) = 0$, and \label{itm:G_GC:reviving BB with revival kernel, i}
   \item for all $\e < \d$, $X_{\t^X_\e} \in \BB_\d(v)$ holds $\PV^X_v$-a.s.\ on $\{ \t^X_\e < \z^X \}$. \label{itm:G_GC:reviving BB with revival kernel, ii}
 \end{enumerate}
 Then $Y$ is a Brownian motion on $\cG$. For all $v \in \cV$, the generator $A^Y$ of $Y$ satisfies for every $f \in \sD(A^Y)$
   \begin{align*}
   c^{v, \infty}_1 f(v) 
   - \sum_{l \in \cL(v)} c^{v,l}_2 f_l'(v) + c^v_3 \, A f(v)
   - \int_{\cG \bs \{v\}} \hspace*{-3pt} \big( f(g) - f(v) \big) \, (c^v_4 + c^{v,\D}_1 \, \kappa^v) (dg)
   = 0.
  \end{align*} 
 If additionally $d(v,x) = +\infty$ holds for every $x \in \supp{\kappa^v}$, then the \FW\ data of $Y$ at $v$ reads
   \begin{align*}
    \big( 0, c^{v,\infty}_1, (c^{v,l}_2)_{l \in \cL(v)}, c^v_3, c^v_4 + c^{v,\D}_1 \, \kappa^v \big).
   \end{align*}
\end{lemma}
\begin{proof}
 By~\cite[Theorem~1.7]{WernerConcat}, $Y$ is a right process and thus strongly Markovian.
 As~$Y_t = X_t$ holds a.s.\ for all $t \leq H_X = H_Y$,
 $Y$ is a Brownian motion on $\cG$.

 Fix $v \in \cV$. We are going to examine the components evolving in the generator of the process~$Y$ and compare 
 them to the respective ones of~$X$.  The components in Feller's theorem~\cite[Theorem~1.2]{WernerMetricA} for the process~$X$ at the vertex~$v$ will be named
 $c^X_1$, $\nu_\e^{X}$, $K^X_\e$, etc., instead of $c^v_1$, $\nu_\e^{v}$, $K^v_\e$.
 The proof will be based on the following two main principles:
 \begin{itemize}
  \item Due to assumption \ref{itm:G_GC:reviving BB with revival kernel, i},
        the processes $Y$ and $X$ are equivalent in a neighborhood of~$v$, more precisely: There exists $\d > 0$
        (e.g.\ being the minimum of $\d$ in assumption~\ref{itm:G_GC:reviving BB with revival kernel, i}
        and the minimal length of all edges incident with $v$) such that
         \begin{align*}
          \forall \e \leq \d: \quad \EV^Y_v ( \t^Y_\e ) = \EV^X_v ( \t^X_\e ),
         \end{align*}
        and for all $n \in \N$, $f_1, \ldots, f_n \in b\sB(\cG)$, $0 \leq t_1 < \ldots < t_n$,
         \begin{align*}
          \hspace*{3em} \PV^Y_v \big( f_1(Y_{t_1}) \, \cdots \, f_1(Y_{t_n}) ; t_n < \t^Y_\d \big) = \PV^X_v \big( f_1(X_{t_1}) \, \cdots \, f_1(X_{t_n}) ; t_n < \t^X_\d \big).
         \end{align*}
        In particular, we have for all $\e < \d$, $A \in \sB(\cG)$:
         \begin{align*}
          \PV^Y_v \big( Y_{\t^Y_\e} \in A \,|\, Y_{\t^Y_\e} \in \BB_\d(v) \big) = \PV^X_v \big( X_{\t^X_\e} \in A \,|\, X_{\t^X_\e} \in \BB_\d(v) \big).
         \end{align*}
  \item Due to assumption \ref{itm:G_GC:reviving BB with revival kernel, ii}, the process $X$ only has jumps from $v$ into $\BB_\d(v)$ or to~$\D$, 
        that is,
         \begin{align*}
          \forall \e < \d: \quad \PV^X_v \big( X_{\t^X_\e} \in \BB_\d(v) \cup \{ \D \} \big) = 1.
         \end{align*}       
        Therefore, $Y$ only can jump into $\comp \BB_\d(v)$ if the underlying process $X$ is killed and
        revived again, which yields
          \begin{align*}
            \PV^Y_v \big( Y_{\t^Y_\e} \in \comp \BB_\d(v) \big) = \PV^X_v \big( X_{\t^X_\e} = \D \big),
          \end{align*}
        and the jump distribution is given by the reviving kernel
           \begin{align*}
            \PV^Y_v \big( Y_{\t^Y_\e} \in A \,|\, Y_{\t^Y_\e} \in \comp \BB_\d(v) \big) =  \kappa^v(A), \quad A \in \sB(\cG).
          \end{align*}         
        Furthermore, the revived process $Y$ is not able to die at all, yielding
           \begin{align*}
            \PV^Y_v \big( Y_{\t^Y_\e} = \D \big) = 0.
          \end{align*}        
 \end{itemize}
 Let $f \in \sD(A^Y)$ and fix $v \in \cV$. The vertex $v$ cannot be a trap for $Y$, as otherwise~$v$ would either be a trap for~$X$, which is impossible by $c^{v,\D}_1 > 0$,
 or $Y$ would be revived at~$v$ when $X$ dies there, which contradicts assumption \ref{itm:G_GC:reviving BB with revival kernel, i}.
 Thus, Dynkin's formula yields
  \begin{align*}
   Af(v)
   & = \lim_{\e \downarrow 0} \frac{\EV^Y_v \big( f(Y_{\t^Y_\e}) \big) - f(v)}{\EV^Y_v(\t^Y_\e)}.
  \end{align*}
 
 We are going to reiterate the steps in the proof of Feller's theorem~\cite[Theorem~1.2]{WernerMetricA} for the process~$Y$, but we will be using
 the normalization factor~$K^X_{\e}$ of~$X$ instead of~$K^Y_{\e}$. This will not pose any problems because $K^X_{\e} \geq K^Y_\e$ holds true,
 which is seen as follows:
 With the scaled exit distributions from $\comp \overline{\BB_\e(v)}$
  \begin{align*}
   \nu^Y_\e(A)
    = \frac{\PV^Y_v \big( Y_{\t^Y_\e} \in A \big)}{\EV^Y_v(\t^Y_\e)}, \quad
   \nu^X_\e(A)
    = \frac{\PV^X_v \big( X_{\t^X_\e} \in A \big)}{\EV^X_v(\t^X_\e)}, \quad A \in \sB \big(\cG \bs \{v\} \big),
  \end{align*}
  for $Y$ and $X$, assumption~\ref{itm:G_GC:reviving BB with revival kernel, i} asserts that for all sufficiently small $\e > 0$,
  \begin{align*}
    K^X_{\e} & = 1 + \frac{\PV^X_v \big( X_{\t^X_\e} = \D \big)}{\EV^X_v(\t^X_\e)} + \int_{\cG \bs \{v\}} \big( 1 - e^{-d(v,g)} \big) \, \nu^X_{\e} (dg) \\
             & = 1 + \frac{\PV^Y_v \big( Y_{\t^Y_\e} \in \comp \BB_\d(v) \big)}{\EV^Y_v(\t^Y_\e)} + \int_{\BB_\d(v) \bs \{v\}} \big( 1 - e^{-d(v,g)} \big) \, \nu^Y_{\e} (dg).
  \end{align*}
 As $\PV^Y_v \big( Y_{\t^Y_\e} = \D \big) = 0$ and 
  \begin{align} \label{eq:G_GC:reviving BB with revival kernel, proof, K_X vs K_Y}
   \frac{\PV^Y_v \big( Y_{\t^Y_\e} \in \comp \BB_\d(v) \big)}{\EV^Y_v(\t^Y_\e)}
   = \int_{ \comp \BB_\d(v) } 1  \, \nu^Y_{\e} (dg)
   \geq \int_{ \comp \BB_\d(v) } \big( 1 - e^{-d(v,g)} \big)  \, \nu^Y_{\e} (dg),
  \end{align}
 we get 
  \begin{align*}
     K^X_{\e}  & \geq 1 + \frac{\PV^Y_v \big( Y_{\t^Y_\e} = \D \big)}{\EV^Y_v(\t^Y_\e) } + \int_{\cG \bs \{v\}} \big( 1 - e^{-d(v,g)} \big) \, \nu^Y_{\e} (dg) \\
             & = K^Y_{\e}.
  \end{align*}
  
 Thus, by following the proof of Feller's theorem (see \cite[Section~3]{WernerMetricA}), we get
  \begin{align*}
   \lim_{\e \downarrow 0} \Big(  f(v) \, \frac{\PV^Y_v \big( Y_{\t^Y_\e} = \D \big)}{\EV^Y_v(\t^Y_\e) \, K^X_{\e}} 
                                 + A f(v) \, \frac{1}{K^X_\e} 
                                 - \int_{\cG \bs \{v\}} \big( f(g) - f(v) \big) \, \frac{\nu^{Y}_\e (dg)}{K^X_\e} \Big)
   = 0.
  \end{align*}  
 However, it is $\PV^Y_v ( Y_{\t^Y_\e} = \D ) = 0$, and the exit distributions of $Y$ decompose into
  \begin{align*}
    \nu^Y_\e(A) = \frac{\PV^Y_v \big(Y_{\t^Y_\e} \in A \cap \BB_\d(v) \big)}{\EV^Y_v(\t^Y_\e)} + \frac{\PV^Y_v \big( Y_{\t^Y_\e} \in A \cap \BB_\d(v)^\comp \big)}{\EV^Y_v(\t^Y_\e)},
  \end{align*}
 with
  \begin{align*}
   \frac{\PV^Y_v \big( Y_{\t^Y_\e} \in A \cap \BB_\d(v) \big)}{\EV^Y_v(\t^Y_\e)}
   & = \frac{\PV^Y_v \big( Y_{\t^Y_\e} \in A \,|\, Y_{\t^Y_\e} \in \BB_\d(v) \big)}{\PV^Y_v\big( Y_{\t^Y_\e} \in \BB_\d(v) \big)} \, \frac{1}{\EV^Y_v(\t^Y_\e)} \\
   & = \frac{\PV^X_v \big( X_{\t^X_\e} \in A \,|\, X_{\t^X_\e} \in \BB_\d(v) \big)}{\PV^X_v\big( Y_{\t^X_\e} \in \BB_\d(v) \big)} \, \frac{1}{\EV^X_v(\t^X_\e)} \\
   & = \frac{\PV^X_v \big( X_{\t^X_\e} \in A \cap \BB_\d(v) \big)}{\EV^X_v(\t^X_\e)} \\
   & = \nu^X_\e(A),
  \end{align*}
 and
  \begin{align*}
   \frac{\PV^Y_v \big( Y_{\t^Y_\e} \in A \cap \BB_\d(v)^\comp \big)}{\EV^Y_v(\t^Y_\e)}
   & = \PV^Y_v \big( Y_{\t^Y_\e} \in A \,|\, Y_{\t^Y_\e} \in \BB_\d(v)^\comp \big) \, \frac{\PV^Y_v \big( Y_{\t^Y_\e} \in \BB_\d(v)^\comp \big) }{\EV^X_v(\t^X_\e)} \\
   & = \kappa^v(A) \, \frac{ \PV^X_v( X_{\t^X_\e} = \D) }{\EV^X_v(\t^X_\e)}.
  \end{align*}
 Therefore, we have
   \begin{align*}
   \lim_{\e \downarrow 0} \Big( &  A f(v) \, \frac{1}{K^X_\e} 
                                 - \int_{\cG \bs \{v\}} \big( f(g) - f(v) \big) \, \frac{\nu^{X}_\e (dg)}{K^X_\e} \\
                               & - \frac{ \PV^X_v( X_{\t^X_\e} = \D) }{\EV^X_v(\t^X_\e) \, K^X_{\e}} \,
                                   \int_{\comp \BB_\d(v)} \big( f(g) - f(v) \big) \, \kappa^v(dg) \Big)
   = 0,
  \end{align*}  
 and knowing that $\frac{1}{K^X_{\e_n}}$, $\frac{\nu^{X}_{\e_n} (dg)}{K^X_{\e_n}}$, $\frac{ \PV^X_v( X_{\t^X_{\e_n}} = \D) }{\EV^X_v(\t^X_{\e_n}) \, K^X_{\e_n}}$
 converge along the same sequence $(\e_n, n \in \N)$ given by Feller's theorem~\cite[Theorem~1.2]{WernerMetricA} for $X$, we conclude that 
  \begin{align*}
   c^{v, \infty}_1 \, f(v) - \sum_{l \in \cL(v)} c^{v,l}_2 f_l'(v) + c^v_3 \, A f(v)
   & - \int_{\cG \bs \{v\}} \big( f(g) - f(v) \big) \, c^v_4 (dg) \\
   & - c^{v, \D}_1 \,
       \int_{\comp \BB_\d(v)} \big( f(g) - f(v) \big) \, \kappa^v(dg)
   = 0.
  \end{align*}  
  
 In case every point in the support of $\kappa^v$ has distance $+\infty$ from $v$,
 equation~\eqref{eq:G_GC:reviving BB with revival kernel, proof, K_X vs K_Y} shows that $K^X_\e = K^Y_\e$ holds true,
 and therefore the above set of boundary conditions at $v$ for $Y$ coincides with the \FW\ data of~$Y$ at $v$. 
\end{proof}

The reader may notice that the resulting boundary data for~$Y$ given in Lemma~\ref{lem:G_GC:reviving BB with revival kernel} 
might not satisfy the normalization condition of the \FW\ data, as given in~\cite[Theorem~1.2]{WernerMetricA},
in case the support of~$\kappa^v$ does not have infinite distance from~$v$. 

\begin{remark} \label{rem:G_GC:reviving BB with revival kernel}
 We observe in Lemma~\ref{lem:G_GC:reviving BB with revival kernel} that the revival of a process upon its death with a revival distribution $\kappa$ only transforms 
 the ``real'' killing parameter $c^{\D}_1$ into an additional jump part $c^{\D}_1 \, \kappa$,
 while leaving the artificial killing portion $c^{\infty}_1$ intact.
 The main explanation is that $c^{\infty}_1$ does not represent the effect of ``killing'' in the sense of proper jumps to the cemetery point $\D$.
 It is rather caused by an explosion of the process, triggered by ever-growing jumps when the process approaches a vertex point,
 and this effect is not transformed by the revival technique.
 
 In the Brownian context, we do not expect any effects which would contribute to $c^{\infty}_1$, and we indeed showed 
 in~\cite[Theorem~1.4]{WernerMetricA} that $c^{\infty}_1$ vanishes for all Brownian motions on star graphs.
 As these processes will form the building blocks of the Brownian motions on the general metric graph,
 the \FW\ data of all processes considered here will satisfy 
  \begin{align*}
    \forall v \in \cV: c^{v,\infty}_1 & = 0.
  \end{align*}
\end{remark}

\section{Gluing the Graphs Together} \label{sec:G_GC:glueing}

We are going to discuss the main construction method, namely the pasting of the subgraphs and their Brownian motions thereon.
As already disclosed in section~\ref{sec:G_GC:agenda},
this technique will compromise several steps:

\subsection{Decomposition of the Graph \texorpdfstring{$\cG$}{G} into \texorpdfstring{$\ttcG^{-1}$}{G''{-1}}, \texorpdfstring{$\ttcG^{+1}$}{G''{+1}}}

Let $\cG = \big( \cV, \cE, \cI, \cLV, \cR \big)$ be a metric graph.
We partition $\cG$ into two graphs by choosing two disjoint, non-empty sets $\cV^{-1}$ and $\cV^{+1}$ with $\cV = \cV^{-1} \uplus \cV^{+1}$,
and decompose the set of edges into
 \begin{align*}
  \cE = \, & \cE^{-1} \uplus \cE^{+1}, \quad \hspace{2.07em} \text{with } \cE^j := \{ e \in \cE: \cLV(e) \in \cV^j \},  \\
  \cI = \, & \cI^{-1} \uplus \cI^{+1} \uplus \cI_s,  \quad \text{with } \cI^j := \{ i \in \cI:  \cLV_{-}(i) \in \cV^j, \cLV_{+}(i) \in \cV^j \},  \\
  \cI_s := \, & \cI_s^{-1} \uplus \cI_s^{+1},  \quad \hspace{2.1em} \text{with } \cI_s^j := \{ i \in \cI: \cLV_{-}(i) \in \cV^j, \cLV_{+}(i) \notin \cV^j \}. 
 \end{align*}
 
As most of the following construction 
will be performed for both partial graphs in parallel, we will always assume that $j \in \{-1, +1\}$ when nothing else is said.
 
We define the metric graphs $\ttcG^{-1}$, $\ttcG^{+1}$ by
 \begin{align*}
  \ttcG^j := \big( \cV^j, \cE^j \cup \cE_s^j, \cI^j, \cLV^j, \cR^j \big),
 \end{align*}
equipped with additional external ``shadow'' edges
 \begin{align*}
  \cE_s^j := \{ e_i^j, i \in \cI_s \}, \quad \text{with } \forall i \in \cI_s: \quad e_i^j \notin \cE \cup \cE_s^{-j} \cup \cI, 
 \end{align*}
where the combinatorial structure and edge lengths of the original graph are naturally transfered to $\ttcG^{-1}$, $\ttcG^{+1}$ by setting
 \begin{align*}
  & \restr{\cLV^j}{\cE^j \cup (I^j \times I^j)} := \restr{\cLV}{\cE^j \cup (I^j \times I^j)}, 
  \quad \cLV^j(e_i^j)
   := \begin{cases}
      \cLV_{-}(i), & i \in \cI_s^j, \\
      \cLV_{+}(i), & i \in \cI_s^{-j},
     \end{cases}  \\
  & \restr{\cR^j}{\cE^j \cup I^j} := \restr{\cR}{\cE^j \cup I^j},
  \quad \restr{\cR^j}{\cE_s^j} := +\infty. 
 \end{align*}
For later use, we also define the ``shadow length'' of an external ``shadow'' edge by 
 \begin{align*}
  \cR_s(e_i^j) := \cR(i), \quad e_i^j \in \cE_s^{-1} \cup \cE_s^{+1}.
 \end{align*}

The excrescent parts of the shadow edges, which will be removed in the following development before gluing both subgraphs together,
are named 
\begin{align*}
 \tcG_s^j := \bigcup_{e \in \cE_s^j} \big( \{e\} \times [\cR_s(e), +\infty) \big). 
\end{align*}

\subsection{Introducing the Brownian Motion \texorpdfstring{$\ttX^{j}$}{X''j} on \texorpdfstring{$\ttcG^{j}$}{G''j}}

Let $\ttX^{-1}$, $\ttX^{+1}$ be Brownian motions on $\ttcG^{-1}$, $\ttcG^{+1}$ respectively, 
which admit the hypotheses of right processes, feature infinite lifetimes, have the \FW\ data
 \begin{align*}
   \big( 0, 0, (p_2^{v,l})_{l \in \ttcL^j(v)}, p_3^v, p_4^v \big)_{v \in \cV^j}, 
 \end{align*}
are continuous inside every edge (cf.~\cite[Theorem~4.3]{WernerStar}), 
and satisfy for all $v \in \cV^j$
 \begin{align} \label{eq:G_GC:gluing, no jumps onto shadow parts}
   \forall \e < \d: \quad \PV_v^j \big( \ttX_{\ttt_\e^j}^j \in \tcG_s^j \big) = 0,
 \end{align}
with $\d := \min \{ \cR_i, i \in \cI_s \}$ and $\ttt_\e^j := \inf \big\{ t \geq 0: d( \ttX_t^j, \ttX_0^j) > \e \big\}$.

By gluing the graphs $\ttcG^{-1}$ and $\ttcG^{+1}$ (and thus the Brownian motions~$\ttX^{-1}$ and~$\ttX^{+1}$ thereon) together,
we are going to show the following main result of this section:

\begin{theorem} \label{theo:G_GC:glueing process}
  There exists a Brownian motion $X$ on $\cG$ with \FW\ data
   \begin{align*}
     \big( c^v_1, (c^{v,l}_2)_{l \in \cL(v)}, c^v_3, c^v_4 \big)_{v \in \cV},
   \end{align*}
  such that for each $v \in \cV$, it holds that $c^v_1 = 0$, $c^v_3 = p^v_3$, $c^v_4 = p^v_4 \circ (\psi^j)^{-1}$,
  with $\psi^j$ being defined by equation~\eqref{eq:G_GC:def psi}, and
  \begin{align*}
   i \in \cI(v): \quad &
   c^{v,i}_2
   = \begin{cases}
      p^{v,i}_2,     & i \in \cI^{-1}(v) \cup \cI^{+1}(v), \\
      p^{v,e^j_i}_2, & i \in \cI_s(v), \text{ with $j \in \{-1,+1\}$ such that $v \in \cV^j$,}
     \end{cases} \\
   e \in \cE(v): \quad &   
   c^{v,e}_2 
   = p^{v,e}_2.
  \end{align*}
\end{theorem}

We construct this process $X$ explicitly via alternating copies of transformed processes $X^{-1}$, $X^{+1}$ of $\ttX^{-1}$, $\ttX^{+1}$.
Before that, we need to kill the original processes $\ttX^{-1}$ and $\ttX^{+1}$ on the excrescent shadow edges and reorientate the remaining parts in order to comply 
with the direction of the original internal edges of $\cG$.

\subsection{Defining \texorpdfstring{$\tX^{j}$}{X'j} by Killing \texorpdfstring{$\ttX^{j}$}{X''j} on \texorpdfstring{$\tcG_s^j$}{G'j}}

Consider the first entry time into $\tcG_s^j$ of the prototype Brownian motion $\ttX^j$ on $\ttcG^j$,
 \begin{align*} 
  \ttt^j := \inf \big\{ t \geq 0: \ttX_t^j \in \tcG_s^j \big\}.
 \end{align*}
We define $\tX^j$ to be the process obtained by killing $\ttX^j$ at the terminal time $\ttt^j$,
 \begin{align*}
  \tX_t^j := \begin{cases}
                \ttX_t^j, & t < \ttt^j, \\
                \D,       & t \geq \ttt^j,
             \end{cases} 
 \end{align*}
on the topological subspace $\tcG^j$ of $\ttcG^j$ given by
 \begin{align*}
  \tcG^j 
  := \ttcG^j \bs \tcG_s^j
   = \cV^j & \cup \bigcup_{l \in \cE^j \cup \cI^j} \big( \{l\} \times (0, \cR_l) \big)
            \cup \bigcup_{e \in \cE_s^j} \big( \{e\} \times \big(0, \cR_s(e) \big) \big). 
 \end{align*}

\begin{lemma} \label{lemma:G_GC:tXj is right process}
 $\tX^j$ is a right process on $\tcG^j$ with lifetime $\ttt^j$.
\end{lemma}
\begin{proof}
  $\ttX^j$ is a right process with infinite lifetime. By employing~\cite[Corollary~12.24]{Sharpe88}, it suffices to observe that
  $\ttt^j$ is the debut of the closed, thus nearly optional set $\tcG_s^j$, 
  and the regular set of the killing time $\ttt^j$ reads
   \begin{align*}
     F := \big\{ g \in \ttcG^j: \PV_g^j(\ttt^j = 0) = 1 \big\} = \tcG_s^j,
   \end{align*}
  as $\ttX^j$ is a right continuous, normal process and $\tcG_s^j$ is closed.
\end{proof}

We would like to point out that the just introduced processes $\tX^j$ are not Brownian motions on a metric graph
in the sense of~\cite[Definitions~2.1, A.1, A.2]{WernerMetricA}
anymore, as $\tcG^j$ is not a metric graph.
Thus, we will not be able to apply any results on Brownian motions for $\tX^j$ in the upcoming development.

\subsection{Letting \texorpdfstring{$X^{j}$}{Xj} be the Mapping of \texorpdfstring{$\tX^{j}$}{X'j} to the Subspace \texorpdfstring{$\cG^{j} \subseteq \cG$}{Gj}}

We need to fit the subspaces $\tcG^j$ of $\ttcG^j$ to the corresponding subspaces of $\cG$.
To this end, we introduce the topological subspaces $\cG^{-1}$, $\cG^{+1}$ of~$\cG$ by
 \begin{align*}
  \cG^j := \cV^j \cup \bigcup_{l \in \cE^j \cup \cI^j \cup \cI_s} \big( \{l\} \times (0, \cR_l) \big),
 \end{align*}
and consider the mapping $\psi^j \colon \tcG^j \rightarrow \cG^j$ defined by
 \begin{equation} \label{eq:G_GC:def psi}
 \begin{aligned}
  & \forall i \in \cI_s, x \in (0, \cR_i): \quad
   \psi^j \big( (e_i^j, x) \big) := 
    \begin{cases}
     (i,x), & i \in \cI_s^j, \\
     (i, \cR_i-x), & i \in \cI_s^{-j},
    \end{cases} \\
  & \psi^j = \id \text{ otherwise}. 
 \end{aligned}
  \end{equation}
Clearly, $\psi^j$ is a bijective mapping, with its inverse $(\psi^j)^{-1} =: \varphi^j \colon \cG^j \rightarrow \tcG^j$ being given by
 \begin{align*}
  & \forall i \in \cI_s, x \in \big( 0, \cR_i \big): \quad
   \varphi^j \big( (i, x) \big) := 
    \begin{cases}
     (e_i^j,x), & i \in \cI_s^j, \\
     (e_i^j, \cR_i-x), & i \in \cI_s^{-j},
    \end{cases} \\
  & \varphi^j = \id \text{ otherwise}.
 \end{align*}
Furthermore, $\psi^j$ is a continuous mapping, as it is continuous inside every edge and its preimages of balls with sufficiently small radius 
around vertices $v \in \cV^j$ coincide with the corresponding balls of $\tcG^j$.

$\tX^j$ is a right process on $\tcG^j$,
$\psi^j$ is a bijective and measurable map from $\tcG^j$ onto $\cG^j$, and $t \mapsto \psi^j(\tX_t^j)$ is right continuous
(as $\psi^j$ is continuous and $t \mapsto \tX_t^j$ is right continuous).
Thus, the following result is a direct consequence of~\cite[Corollary~(13.7)]{Sharpe88}:
 
\begin{lemma} \label{lemma:G_GC:Xj is right process}
 The process $X^j := \psi^j(\tX^j)$, resulting from the state space mapping of $\tX^j$ by $\psi^j$,
 is a right process on $\psi^j(\tcG^j) = \cG^j$ with lifetime $\z^j = \widetilde{\z}{}^j = \ttt^j$.
\end{lemma}

\subsection{Constructing \texorpdfstring{$X$}{X} as Alternating Copies Process of \texorpdfstring{$X^{-1}$}{X{-1}}, \texorpdfstring{$X^{+1}$}{X{+1}}}

We apply the technique of~\cite{WernerConcat} to define 
the process $X$ obtained by forming alternating copies of~$X^{-1}$ and~$X^{+1}$ 
via the transfer kernels $K^{-1}$ and $K^{+1}$, given by
 \begin{equation} \label{eq:G_GC:transfer kernels}
 \begin{aligned}
  K^{-1} & := \sum_{i \in I_s^{-1}} \e_{\cLV_{+}(i)} \, \1_{\{i\}} \big( \p^1(X^{-1}_{\z^{-1}-}) \big)
           + \sum_{i \in I_s^{+1}} \e_{\cLV_{-}(i)} \, \1_{\{i\}} \big( \p^1(X^{-1}_{\z^{-1}-}) \big), \\
  K^{+1} & := \sum_{i \in I_s^{-1}} \e_{\cLV_{-}(i)} \, \1_{\{i\}} \big( \p^1(X^{+1}_{\z^{+1}-}) \big)
           + \sum_{i \in I_s^{+1}} \e_{\cLV_{+}(i)} \, \1_{\{i\}} \big( \p^1(X^{+1}_{\z^{+1}-}) \big).           
 \end{aligned}
 \end{equation}
That is, the transfer kernels implement the following rules for $j \in \{-1, +1\}$:
\begin{enumerate}
 \item $X$ is revived as $X^{+1}$ at $v = \cLV_{-j}(i)$, if $X^{-1}$ dies on $i \in \cI_s^j$; 
 \item $X$ is revived as $X^{-1}$ at $v = \cLV_{j}(i)$, if $X^{+1}$ dies on $i \in \cI_s^j$. 
\end{enumerate}
For later use, we give the following combined formula of the above definitions for the transfer kernels $K^j$, $j \in \{-1, +1\}$:
\begin{align}  \label{eq:G_GC:transfer kernels summarized}
 K^j 
  = k^j(i)
 := \begin{cases}
        \e_{\cLV_+(i)}, & i \in \cI_s^j, \\
        \e_{\cLV_-(i)}, & i \in \cI_s^{-j},
       \end{cases}
 \quad \text{for} \quad
 i := \p^1(X^{j}_{\z^{j}-}).
\end{align} 
 
\begin{lemma}
 $K^{j}$ is a transfer kernel from $X^j$ to $E^{-j}$.
\end{lemma}
\begin{proof}
 With probability $1$, the process $\ttX^j$ cannot realize $\ttt^j$ through a direct jump from any vertex $v \in \cV^j$:
 Otherwise, this would imply $\PV^j_v \big( \ttX^j_{\ttt^j_\e} \in \tcG^j_s \big) > 0$, as~$\ttt^j \geq \ttt^j_\e$ holds for $\e < \d$, contradicting 
   our fundamental assumption~\eqref{eq:G_GC:gluing, no jumps onto shadow parts}.
 Furthermore, $\ttX^j$ is continuous on every edge, so
  $\ttX^j_{\ttt^j-}$ exists and is equal to $\ttX^j_{\ttt^j}$.
 Thus, 
  \begin{align*}
   \tX^j_{\z^j-} = \lim_{t \upuparrows \z^j} \tX^j_t = \lim_{t \upuparrows \ttt^j} \ttX^j_t = \ttX^j_{\ttt^j}
  \end{align*}
 exists in $\big\{ \big( e, \cR_s(e) \big), e \in \cE^j_s \big\}$, 
 and 
  \begin{align} \label{eq:G_GC:death point of X}
   \p^1 \big( X^j_{\z^j-} \big) = \p^1 \big( \psi^j (\tX^j_{\z^j-}) \big) = \p^1 \big(\psi^j (\ttX^j_{\ttt^j})\big) 
  \end{align}
 exists in $\cI_s$. Therefore, $\p^1 \big( X^j_{\z^j-} \big) \in \sF^j_{[\z^j-]}$, so
 $K^j$ is indeed a probability kernel $K$ from $(\O^j, \sF^j_{[\z^j-]})$ to $(E^{-j}, \sE^{-j})$, that is, a transfer kernel.
\end{proof}

Let
 \begin{itemize}
  \item $\t^{-1}_{-1}$ be the first entry time of $X^{-1}$ into $\cG^{-1} \bs \cG^{+1}$, $\z^{-1}$ the lifetime of $X^{-1}$,
  \item $\t^{+1}_{+1}$ be the first entry time of  $X^{+1}$ into $\cG^{+1} \bs \cG^{-1}$, $\z^{+1}$ the lifetime of $X^{+1}$. 
 \end{itemize}    
Then, according to~\cite[Theorem~1.6]{WernerConcat}, $X$ is a right process on $\cG = \cG^{-1} \cup \cG^{+1}$ 
in case the following conditions hold true for all $g \in \cG^{-1} \cap \cG^{+1}$, $f \in b\sB(\cG)$, $h^{-1} \in b\sB(\cG^{-1})$, $h^{+1} \in b\sB(\cG^{+1})$:
 \begin{enumerate}
  \item $ \displaystyle \label{eq:G_GC:glueing, conc, i}
                  \EV^{-1}_g \Big( \int_0^{\t^{-1}_{-1}} e^{-\a t} \, f(X^{-1}_t) \, dt \Big) 
                = \EV^{+1}_g \Big( \int_0^{\t^{+1}_{+1}} e^{-\a t} \, f(X^{+1}_t) \, dt \Big) $;
  \item $ \displaystyle \label{eq:G_GC:glueing, conc, ii}
                  \EV^{-1}_g \big( e^{-\a \t^{-1}_{-1}} \, h^{-1}(X^{-1}_{\t^{-1}_{-1}}); \, \t^{-1}_{-1} < \z^{-1} \big) 
                = \EV^{+1}_g \big( e^{-\a \z^{+1}} \, K^{+1} h^{-1};                      \, \z^{+1} < \t^{+1}_{+1} \big) $, \\
              $ \displaystyle
                  \EV^{+1}_g \big( e^{-\a \t^{+1}_{+1}} \, h^{+1}(X^{+1}_{\t^{+1}_{+1}}); \, \t^{+1}_{+1} < \z^{+1} \big) 
                = \EV^{-1}_g \big( e^{-\a \z^{-1}} \, K^{-1} h^{+1};                      \, \z^{-1} < \t^{-1}_{-1} \big) $.
 \end{enumerate}

We are preparing the proof of these equalities: By construction, we have
 \begin{align*}
  \cG^{-1} \cap \cG^{+1}
  & = \bigcup_{i \in \cI_s} \big( \{i\} \times (0, \cR_i) \big), \\
  \cG^{j} \bs \cG^{-j} 
  & = \cV^j \cup \bigcup_{l \in \cE^j \cup \cI^j} \big( \{l\} \times (0, \cR_l) \big). 
 \end{align*}
By using the definition of $X^j$ and observing that $\varphi^j(\cG^{j} \bs \cG^{-j}) = \cG^{j} \bs \cG^{-j}$, we get
 \begin{align*}
  \t_j^j
  & = \inf \big\{ t \geq 0: X_t^j \in \cG^{j} \bs \cG^{-j} \big\} \\
  & = \inf \big\{ t \geq 0: \psi^j(\tX_t^j) \in \cG^{j} \bs \cG^{-j} \big\}  \\
  & = \inf \big\{ t \geq 0: \tX_t^j \in \cG^{j} \bs \cG^{-j} \big\}.
 \end{align*}
The process $\tX^j$ was constructed by killing $\ttX^j$ at $\ttt^j$. Thus, by introducing the first exit times of $\ttX^j$ from the shadow edges
 \begin{align*}
  \ttt_j^j := \inf \big\{ t \geq 0: \ttX_t^j \in \cG^{j} \bs \cG^{-j} \big\} 
            = \inf \big\{ t \geq 0: \ttX_t^j \notin \textstyle{ \bigcup_{i \in \cI_s} } \big( \{ e^j_i \} \times (0, \infty) \big) \big\}, 
 \end{align*}
we obtain the relation
 \begin{align} \label{eq:G_GC:first exit times backtrack}
  \t_j^j \wedge \z^j = \ttt^j_j \wedge \ttt^j.
 \end{align}

Turning to the actual proof of \ref{eq:G_GC:glueing, conc, i} and \ref{eq:G_GC:glueing, conc, ii},
let $g \in \cG^{-1} \cap \cG^{+1}$, that is, $g = (i,x)$ for some $i \in \cI_s$, $x \in ( 0, \cR_i )$. 
Choose $j \in \{-1, +1\}$ such that $i \in \cI_s^j$. By tracing $X^j$ back to $\ttX^j$ and employing that the latter is a Brownian motion on $\ttcG^j$, 
\cite[Lemma~2.4 and Corollary~2.10]{WernerMetricA} yield
\begin{align}
 \EV^{-j}_g \Big( \int_0^{\t^{-j}_{-j}} e^{-\a t} \, f \big(X^{-j}_t \big) \, dt \Big)  
 & = \EV^{-j}_{(\psi^{-j})^{-1}(g)} \Big( \int_0^{\t^{-j}_{-j}} e^{-\a t} \, f \big( \psi^{-j} \big(\ttX^{-j}_t, \ttt^{-j} \big) \big) \, dt \Big) \notag \\
 & = \EV^{-j}_{(e_i^{-j}, \cR(i)-x)} \Big( \int_0^{\ttt^{-j}_{-j} \wedge \ttt^{-j}} e^{-\a t} \, f \big( \psi^{-j}(\ttX^{-j}_t) \big) \, dt \Big) \label{eq:G_GC:glueing, conc, i I} \\
 & = \EV^B_{\cR(i)-x} \Big( \int_0^{\t_0 \wedge \t_{\cR(i)}} e^{-\a t} \, f \big( \psi^{-j}(e_i^{-j}, B_t) \big) \, dt \Big) \notag  \\
 & = \EV^B_{\cR(i)-x} \Big( \int_0^{\t_0 \wedge \t_{\cR(i)}} e^{-\a t} \, f \big( i, \cR(i)-B_t \big) \, dt \Big),   \notag
\end{align}
and analogously 
\begin{equation} \label{eq:G_GC:glueing, conc, i II}
\begin{aligned}
 \EV^{j}_g \Big( \int_0^{\t^{j}_{j}} e^{-\a t} \, f \big(X^{j}_t \big) \, dt \Big)  
 & = \EV^{j}_{(e_i^{j}, x)} \Big( \int_0^{\ttt^{j}_{j} \wedge \ttt^{j}} e^{-\a t} \, f \big( \psi^{j}(\ttX^{j}_t) \big) \, dt \Big) \hspace{3em}  \\
 & = \EV^B_{x} \Big( \int_0^{\t_0 \wedge \t_{\cR(i)}} e^{-\a t} \, f \big( i, B_t \big) \, dt \Big).
\end{aligned}
\end{equation}
By the spatial homogeneity and reflection invariance of the one-dimensional Brownian motion~$B$, we have
\begin{align*}
 \EV^B_{\cR(i)-x} \Big( \int_0^{\t_0 \wedge \t_{\cR(i)}} e^{-\a t} \, f \big( i, \cR(i)-B_t \big) \, dt \Big)  
 & = \EV^B_{x} \Big( \int_0^{\t_0 \wedge \t_{\cR(i)}} e^{-\a t} \, f \big( i, B_t \big) \, dt \Big),
\end{align*}
which proves the equality of~\eqref{eq:G_GC:glueing, conc, i I} and~\eqref{eq:G_GC:glueing, conc, i II}, and thus concludes~\ref{eq:G_GC:glueing, conc, i}.

Coming to \ref{eq:G_GC:glueing, conc, ii}, we will prove both assertions simultaneously, as they only differ in the initial process. 
Let $j \in \{-1, +1\}$. We start by reducing the first expectation to~$\ttX^j$, and obtain with the help of equation~\eqref{eq:G_GC:first exit times backtrack} 
the identity
\begin{align*}
 & \EV^{-j}_g \big( e^{-\a \t^{-j}_{-j}} \, h^{-j} \big(X^{-j}_{\t^{-j}_{-j}} \big); \, \t^{-j}_{-j} < \z^{-j} \big)  \\
 & = \EV^{-j}_{(\psi^{-j})^{-1}(g)} \Big( e^{-\a \ttt^{-j}_{-j}} \, h^{-j} \big(\psi^{-j}(\ttX^{-j}_{\ttt^{-j}_{-j}}) \big); \, \ttt^{-j}_{-j} < \ttt^{-j} \Big), 
\end{align*}
where $(\psi^{-j})^{-1}(g) = (e_i^{-j}, \cR_i-x)$ or $(\psi^{-j})^{-1}(g) = (e_i^{-j}, x)$ depending on whether $i \in \cI_s^j$ or $ i \in \cI_s^{-j}$.
For all that follows, we define for any $g \in \ttcG^j$ the first hitting time~$\ttH^j_g$ of the set~$\{g\}$ by the process~$\ttX^j$.
By the continuity of $\ttX^j$ inside the edges,
we see that $\PV^j_{(e^j_i,y)}$-a.s.\ for any $i \in \cI_s$, 
the relation
 \begin{align*}
  \ttt_j^{j} = \ttH^j_{v} \quad \text{on $\big\{ \ttt^{j}_{j} < \ttt^{j} \big\} = \big\{ \ttH^j_{v} < \ttH^j_{(e^j_i, \cR_s(e^j_i))} \big\}$}
 \end{align*}
holds true with $v := \cLV(e^j_i)$, so we have
 \begin{align*}
  \psi^{-j}(\ttX^{-j}_{\ttt^{-j}_{-j}})
  = \cLV(e_i^{-j})
  = \begin{cases}
     \cLV_+(i), & i \in \cI_s^j, \\
     \cLV_-(i), & i \in \cI_s^{-j}.
    \end{cases}
 \end{align*}
Therefore, we get
\begin{align*}
 & \EV^{-j}_g \big( e^{-\a \t^{-j}_{-j}} \, h^{-j} \big(X^{-j}_{\t^{-j}_{-j}} \big); \, \t^{-j}_{-j} < \z^{-j} \big)  \\
 & = \EV^{-j}_{(\psi^{-j})^{-1}(g)} \Big( e^{-\a \ttt^{-j}_{-j}} \, h^{-j} \big(\psi^{-j}(\ttX^{-j}_{\ttt^{-j}_{-j}}) \big); \, \ttt^{-j}_{-j} < \ttt^{-j} \Big)  \\
 & = 
  \begin{cases}
   \EV^{-j}_{(e_i^{-j}, \cR(i)-x)} \Big( e^{-\a \ttH^j_{v}} \, h^{-j} \big( \cLV_+(i) \big); \, \ttH^{-j}_{v} < \ttH^{-j}_{(e^{-j}_i, \cR_s(e^{-j}_i))} \Big), & i \in \cI_s^j,  \\
   \EV^{-j}_{(e_i^{-j}, x)}        \Big( e^{-\a \ttH^j_{v}} \, h^{-j} \big( \cLV_-(i) \big); \,  \ttH^{-j}_{v} < \ttH^{-j}_{(e^{-j}_i, \cR_s(e^{-j}_i))} \Big), & i \in \cI_s^{-j}. 
  \end{cases}
\end{align*}
 But $\ttX^{-j}$ is a Brownian motion on $\ttcG^{-j}$, so \cite[Corollary~2.10, Remark~2.9]{WernerMetricA}
 together with $\cR_s(e^{-j}_i) = \cR(i)$ yield
\begin{equation} \label{eq:G_GC:glueing, conc, ii I}
\begin{aligned}
 & \EV^{-j}_g \big( e^{-\a \t^{-j}_{-j}} \, h^{-j} \big(X^{-j}_{\t^{-j}_{-j}} \big); \, \t^{-j}_{-j} < \z^{-j} \big)  \\
 & = 
  \begin{cases}
   h^{-j} \big( \cLV_+(i) \big) \, \EV^B_{\cR(i)-x} \big( e^{-\a \t_0}; \, \t_0 < \t_{\cR(i)} \big), & i \in \cI_s^j, \\
   h^{-j} \big( \cLV_-(i) \big) \, \EV^B_{x}        \big( e^{-\a \t_0}; \, \t_0 < \t_{\cR(i)} \big), & i \in \cI_s^{-j}.
  \end{cases}
\end{aligned}
\end{equation}

Next, we employ the same techniques as above in order to compute the right-hand sides of \ref{eq:G_GC:glueing, conc, ii}.
Equations~\eqref{eq:G_GC:death point of X} and~\eqref{eq:G_GC:first exit times backtrack} give
\begin{align*}
 & \EV^{j}_g \big( e^{-\a \z^{j}} \, K^{j} h^{-j}; \,  \z^{j} < \t^{j}_{j} \big)   \\
 & = \EV^{j}_{(\psi^{j})^{-1}(g)} \Big( e^{-\a \ttt^{j}} \, k^{j} \big( \p^1\big( \psi^j(\ttX^j_{\ttt^j}) \big) \big) \, h^{-j}; \, \ttt^{j} < \ttt^{j}_{j} \Big).
\end{align*}
We observe that
 \begin{align*}
  \ttt^j = \ttH^j_{(e^j_i, \cR_s(e^j_i))} \quad \text{on $\big\{ \ttt^{j} < \ttt^{j}_{j} \big\} = \big\{ \ttH^j_{(e^j_i, \cR_s(e^j_i))} < \ttH^j_{v} \big\}$},
 \end{align*}
as $\ttt^{j}_{j} \leq \ttt^{j}$ in case $\ttt^j = \ttH^j_{(e^j_k, \cR_s(e^j_k))}$ for some other $k \neq i$.
Thus, we have 
 \begin{align*}
  \p^1 \big( \psi^j(\ttX^j_{\ttt^j}) \big)  = \p^1 \big( \psi^j\big( (e^j_i, \cR_s(e^j_i)) \big)\big) \quad \text{$\PV_{(e^j_i, x)}$-a.s.\ on  $\{ \ttt^{j} < \ttt^{j}_{j} \}$},
 \end{align*}
and because $\psi^j$ maps $e^j_i$ to $i$,
the definition of the transfer kernel $K^j$, which was summarized in equation~\eqref{eq:G_GC:transfer kernels summarized}, gives 
\begin{align*}
 K^{j} 
 = k^j \big(\p^1 \big( \psi^j\big( (e^j_i, \cR_s(e^j_i)) \big) \big)
 = \begin{cases}
    \e_{\cLV_+(i)}, & i \in \cI_s^j, \\
    \e_{\cLV_-(i)}, & i \in \cI_s^{-j}.
   \end{cases}
\end{align*}
This results in
\begin{equation} \label{eq:G_GC:glueing, conc, ii II}
\begin{aligned}
 & \EV^{j}_g \big( e^{-\a \z^{j}} \, K^{j} h^{-j}; \,  \z^{j} < \t^{j}_{j} \big)   \\
 & = \EV^{j}_{(\psi^{j})^{-1}(g)} \Big( e^{-\a \ttt^{j}} \, k^{j} \big( \p^1\big( \psi^j(\ttX^j_{\ttt^j}) \big) \big) \, h^{-j}; \, \ttt^{j} < \ttt^{j}_{j} \Big)  \\
 & =
  \begin{cases}
   \EV^j_{(e_i^j, x)}        \Big( e^{-\a \ttH^j_{(e^j_i, \cR_s(e^j_i))}} \, h^{-j} \big( \cLV_+(i) \big); \, \ttH^j_{(e^j_i, \cR_s(e^j_i))} < \ttH^j_{v} \Big), & i \in \cI_s^j,  \\
   \EV^j_{(e_i^j, \cR(i)-x)} \Big( e^{-\a \ttH^j_{(e^j_i, \cR_s(e^j_i))}} \, h^{-j} \big( \cLV_-(i) \big); \, \ttH^j_{(e^j_i, \cR_s(e^j_i))} < \ttH^j_{v} \Big), & i \in \cI_s^{-j} 
  \end{cases} \\
 & =
  \begin{cases}
   h^{-j} \big( \cLV_+(i) \big) \, \EV^B_{x} \big( e^{-\a \t_{\cR(i)}};        \, \t_{\cR(i)} < \t_0 \big), & i \in \cI_s^j, \\
   h^{-j} \big( \cLV_-(i) \big) \, \EV^B_{\cR(i)-x} \big( e^{-\a \t_{\cR(i)}}; \, \t_{\cR(i)} < \t_0 \big), & i \in \cI_s^{-j}.
  \end{cases}
\end{aligned}
\end{equation}
Now, the first passage time formulas for the one-dimensional Brownian motion~$B$ (cf.~\cite[Section~1.7]{ItoMcKean74}) give 
\begin{align*}
 \EV^B_{\cR(i)-x} \big( e^{-\a \t_0}; \t_0 < \t_{\cR(i)} \big)
 & = \frac{\sinh \big( \sqrt{2\a} \, x \big)}{\sinh \big( \sqrt{2\a} \, \cR(i) \big)} 
   = \EV^B_{x} \big( e^{-\a \t_{\cR(i)}} ;  \t_{\cR(i)} < \t_0 \big), \\
 \EV^B_{x} \big( e^{-\a \t_0}; \t_0 < \t_{\cR(i)} \big)
 & = \frac{\sinh \big( \sqrt{2\a} \, (\cR(i) - x) \big)}{\sinh \big( \sqrt{2\a} \, \cR(i) \big)} 
   = \EV^B_{\cR(i)-x} \big( e^{-\a \t_{\cR(i)}} ;  \t_{\cR(i)} < \t_0 \big).
\end{align*}
A comparison of the equations~\eqref{eq:G_GC:glueing, conc, ii I} and~\eqref{eq:G_GC:glueing, conc, ii II} then proves the equalities in~\ref{eq:G_GC:glueing, conc, ii}.

We have shown that the conditions of~\cite[Theorem~1.6]{WernerConcat} are fulfilled and thus have proved:
\begin{lemma} \label{lem:G_GC:glued process is right process}
 The process $X$ which is obtained by forming alternating copies of~$X^{-1}$ and~$X^{+1}$
 via the transfer kernels~$K^{-1}$ and~$K^{+1}$, as defined by equation~\eqref{eq:G_GC:transfer kernels},
 is a right process on~$\cG^{-1} \cup \cG^{+1} = \cG$.
\end{lemma}

\subsection{Proving that \texorpdfstring{$X$}{X} is a Brownian Motion on \texorpdfstring{$\cG$}{G}}

As just seen, $X$ is a right process and therefore a strong Markov process on $\cG$.
In regard to \cite[Theorem~2.5]{WernerMetricA}, it suffices to analyze the stopped resolvent and the exit behavior from any edge
in order to show that $X$ is indeed a Brownian motion on $\cG$:

\begin{lemma} \label{lem:G_GC:glued process is BM}
 $X$ is a Brownian motion on $\cG$.
\end{lemma}
\begin{proof}
For mutual edges $i \in \cI_s$, we choose $j \in \{-1, +1\}$ such that $i \in \cI^j_s$.
Then we have $X_t = X^j_t$ for all $t < \t^j_j$ and $X_{R^1} \in \cLV(i)$, $\PV_{(i,x)}$-a.s., 
for the first revival time $R^1 = \inf \big\{ t \geq 0: X_t \in \cG^{-j} \bs \cG^{j} \big\}$. 
Therefore, $H_X = \t^j_j \wedge R^1$
holds true, and with equation~\eqref{eq:G_GC:glueing, conc, i II} we get
 \begin{align*}
 \EV_{(i,x)} \Big( \int_0^{H_X} e^{-\a t} \, f(X_t) \, dt \Big) 
 & = \EV^{j}_{(i,x)} \Big( \int_0^{\t^{j}_{j} \wedge \z^j} e^{-\a t} \, f \big(X^{j}_t \big) \, dt \Big)  \\
 & = \EV^B_{x} \Big( \int_0^{\t_0 \wedge \t_{\cR(i)}} e^{-\a t} \, f \big( i, B_t \big) \, dt \Big) \\
 & = \EV^B_{x} \Big( \int_0^{H_B} e^{-\a t} \, f(i, B_t) \, dt \Big).
 \end{align*}
For non-mutual edges $l \notin \cI_s$, on the other hand, choose $j \in \{-1, +1\}$ such that~$(l,x) \in \ttcG^j$.
Then $X^j_t = \ttX^j_t$ holds for all $t < \t^j_j$, $\PV_{(l,x)}$-a.s., 
and as $\ttX^j$ is itself a Brownian motion on~$\ttcG^j$, the above identity follows immediately.

Coming to the exit distribution from an edge, the identity
 \begin{align*}
  \PV_{(l,x)} \circ \big( H_X, X_{H_X} \big)^{-1} = \PV^B_{x} \circ \big( H_B, (l, B_{H_B}) \big)^{-1}
 \end{align*}
follows for edges $l \notin \cI_s$ from the corresponding property of $\ttX^{-1}$ or $\ttX^{+1}$  by~\cite[Theorem~2.5]{WernerMetricA}. 
In case $i \in \cI_s$, we choose $j \in \{-1, +1\}$ with $i \in \cI_s^j$.
By employing equations~\eqref{eq:G_GC:glueing, conc, ii I}, \eqref{eq:G_GC:glueing, conc, ii II} 
and $H_X = \t^j_j \wedge R^1$ $\PV_{(i,x)}$-a.s.,
we get for all $\a > 0$, $h \in b\sB(\cG)$
 \begin{align*}
  & \EV_{(i,x)} \big( e^{-\a H_X} h(X_{H_X}) \big) \\
  & = \EV^j_{(i,x)} \big( e^{-\a \t^{j}_{j}} h(X^j_{\t^{j}_{j}}); \, \t^{j}_{j} < \z^j \big)
    + \EV^j_{(i,x)} \big( e^{-\a \z^j} K^j g; \, \z^j < \t^{j}_{j} \big) \\
  & = \EV^B_{x} \big( e^{-\a \t_0} \, h \big( \cLV_-(i) \big); \, \t_0 < \t_{\cR(i)} \big)
    + \EV^B_{x} \big( e^{-\a \t_{\cR(i)}} \, h \big( \cLV_+(i) \big); \,         \t_{\cR(i)} < \t_0 \big) \\
  & = \EV^B_{x} \big( e^{-\a H_B} h(i, B_{H_B}) \big),
 \end{align*}
which results in
 \begin{align*}
  \PV_{(i,x)} \circ \big( H_X, X_{H_X} \big)^{-1}
  & = \PV^B_{x} \circ \big( H_B, (i, B_{H_B}) \big)^{-1}. \qedhere
 \end{align*}
\end{proof}

\subsection{Computing the \FW\ Data of \texorpdfstring{$X$}{X}}

The \FW\ data of~$X$, as given in~\cite[Theorem~1.2]{WernerMetricA}, is derived from its exit distributions from any arbitrarily small neighborhood of each vertex. 
$X$ is constructed via alternating copies of $X^{-1}$ and $X^{+1}$, so we first need to analyze their respective exit behavior.
To this end, we consider the exit times of $X^j$
 \begin{align*}
   \t^j_\e := & \inf \big\{ t \geq 0: d \big( X^j_t, X^j_0 \big) > \e \big\} 
 \end{align*}
together with the exit distributions $X^j_{\t^j_\e}$ for all small $\e > 0$. 
As we only have information on $\ttX^j$, we need to trace back the required data to these original processes.
Fix $v \in \cV$ and choose $j \in \{-1, +1\}$ such that $v \in \cV^j$, and let 
 \begin{align*}
   \ttt^j_\e := & \inf \big\{ t \geq 0: d \big( \ttX^j_t, \ttX^j_0 \big) > \e \big\}.
 \end{align*} 

Using the definition of $X^j$ and the isometric property of~$\psi^j$, we get for all $\e > 0$
 \begin{align*}
  \t^j_\e 
  & = \, \inf \big\{ t \geq 0: d \big( \psi^j(\tX^j_t), \psi^j(\tX^j_0) \big) > \e \big\} \\
  & = \, \inf \big\{ t \geq 0: d \big( \tX^j_t, \tX^j_0 \big) > \e \big\} \\
  & =:  \tt^j_\e.
 \end{align*}
By its definition, $\tX^j_t = \ttX^j_t$ holds for all $t < \ttt^j$, and as $\comp \BB_\e(v) \supseteq \tcG^j_s$, we obtain
 \begin{align*}
  \forall \e < \d: \quad \ttt^j_\e \leq \ttt^j \quad \text{$\PV^j_v$-a.s.}\,.
 \end{align*}
More precisely, we even get 
 \begin{align*}
  \forall \e < \d: \quad \ttt^j_\e < \ttt^j, \quad \text{if} \quad \ttt^j_\e \neq +\infty, \quad \text{$\PV^j_v$-a.s.},
 \end{align*}
because
 \begin{align*}
  \PV^j_v \big( \ttt^j_\e = \ttt^j, \ttt^j_\e < +\infty \big)
  & = \PV^j_v \big( \ttt^j_\e = \ttt^j, \ttX_{\ttt^j} \in \tcG^j_s, \ttt^j_\e < +\infty \big) \\
  & = \PV^j_v \big( \ttt^j_\e = \ttt^j, \ttX_{\ttt^j_\e} \in \tcG^j_s, \ttt^j_\e < +\infty \big) \\
  & \leq \PV^j_v \big( \ttX_{\ttt^j_\e} \in \tcG^j_s \big) = 0.
 \end{align*}
Therefore, we see that for all $\e < \d$,
 \begin{align*}
  \ttt^j_\e 
  & = \inf \big\{ t \in [0, \ttt^j): d(\ttX^j_t, \ttX^j_0) > \e \big\} \wedge \ttt^j \\
  & = \inf \big\{ t \in [0, \ttt^j): d(\tX^j_t, \tX^j_0) > \e \big\} \wedge \ttt^j \\
  & = \inf \big\{ t \geq 0: d(\tX^j_t, \tX^j_0) > \e \big\},
 \end{align*}
where we used that $\tX^j$ is a subprocess of $\ttX^j$ with lifetime $\ttt^j$, that is 
 \begin{align*}
   d(\tX^j_{\ttt^j}, \tX^j_0) = d( \D, \tX^j_0) = +\infty > \e.
 \end{align*}

We have thus shown:
\begin{lemma} \label{lem:G_GC:Feller data, exit times}
 Let $v \in \cV^j$. For all $\e < \d$, it holds $\PV^j_v$-a.s.\ that 
  \begin{align*}
   \t^j_\e = \tt^j_\e = \ttt^j_\e,
  \end{align*}
 and
  \begin{align*}
   \ttt^j_\e < \ttt^j, \quad \text{if} \quad \ttt^j_\e < +\infty.
  \end{align*}
\end{lemma}

\begin{corollary} \label{cor:G_GC:Feller data, exit distributions}
 For all $v \in \cV^j$, $\e < \d$, the exit distribution of $X^j$ is given by
  \begin{align*}
   X^j_{\t^j_\e} = \begin{cases}
                    \psi^j( \ttX^j_{\ttt^j_\e} ), & \ttt^j_\e < +\infty, \\
                    \D,                           & \ttt^j_\e = +\infty.
                   \end{cases}
  \end{align*}
\end{corollary}

We are ready to compute the \FW\ data of $X$.
By Lemmas~\ref{lem:G_GC:Feller data, exit times} and~\ref{lemma:G_GC:Xj is right process}, we have for all $\e < \d$
 \begin{align*}
  \t^j_\e = \ttt^j_\e < \ttt^j = \z^j \quad \text{ on $\{ \z^j < +\infty \}$},
 \end{align*}
so $\t^j_\e < \z^j$ a.s.\ holds. On the other hand, $X_t = X^j_t$ holds for all $t < R^1 = \z^j$
(more formally, $X^j_t (\o_i) = X_t \big( (\o_1, \o_2, \ldots) \big)$ with $i = 1$ if $j = -1$, and $i= 2$ if $j = +1$)
by the construction of $X$, yielding
 \begin{align*}
   \PV_v \circ \big( {\t_\e}, X_{\t_\e} \big)^{-1} = \PV^j_v \circ \big({\t^j_\e},  X^j_{\t^j_\e} \big)^{-1}.
 \end{align*}

 Thus, if $v$ is not a trap, then  
 $\ttt^j_\e < +\infty$ holds $\PV^j_v$-a.s.\ for all sufficiently small~$\e > 0$ (see~\cite[Lemma~B.1]{WernerMetricA}), 
 and therefore $\t_\e < +\infty$ holds $\PV_v$-a.s.\ as well.
 By using the notations of~\cite[Theorem~1.2]{WernerMetricA} and backtracking $X$ to $\ttX^j$, we compute 
 for $\e < \d$, for all~$A \in \sB\big(\cG \bs \{v\} \big)$:
 \begin{align*}
  \nu^v_\e(A) 
  & = \frac{\PV_v \big( X_{\t_\e} \in A \big)}{\EV_v(\t_\e)} 
   = \frac{\PV^j_v \big( X^j_{\t^j_\e} \in A \big)}{\EV^j_v(\t^j_\e)} 
   = \frac{\PV^j_v \big( \psi^j( \ttX^j_{\ttt^j_\e} ) \in A \big)}{\EV^j_v(\ttt^j_\e)} 
   = \ttnu^{j,v}_\e \big( (\psi^j)^{-1} (A) \big),
 \end{align*}
where we naturally extend, here and in all that follows, the mapping $\psi^j \colon \tcG^j \rightarrow \cG^j$ to $\psi^j \colon \tcG^j \rightarrow \cG$.
This gives
 \begin{align*}
  K^v_\e 
  & = 1 + \frac{\PV_v \big( X_{\t_\e} = \D \big)}{\EV_v(\t_\e)} + \int_{\cG \bs \{v\}} \big( 1 - e^{-d(v,g)} \big) \, \nu^v_\e(dg) \\
  & = 1 + \frac{\PV^j_v \big( \ttX^j_{\ttt^j_\e} = \D \big)}{\EV^j_v(\ttt^j_\e)} + \int_{\tcG^j \bs \{v\}} \big( 1 - e^{-d(v,\psi^j(g))} \big) \, \ttnu^{j,v}_\e(dg) \\
  & = 1 + \frac{\PV^j_v \big( \ttX^j_{\ttt^j_\e} = \D \big)}{\EV^j_v(\ttt^j_\e)} + \int_{\ttcG^j \bs \{v\}} \big( 1 - e^{-d(v,g)} \big) \, \ttnu^{j,v}_\e(dg) \\
  & = \ttK^{j,v}_\e,
 \end{align*}
because $\psi^j$ is an isometry with  $\psi^j(v) = v$, and as $\ttnu^v_\e \big( \ttcG^j \bs \tcG^j \big) = 0$ holds due to the assumption~\eqref{eq:G_GC:gluing, no jumps onto shadow parts}.
Renormalization yields, again because $\psi$ is an isometry,
\begin{equation} \label{eq:G_GC:mu relation}
 \begin{aligned}
 \forall A \in \sB\big(\cG \bs \{v\} \big): \quad
  \mu^v_\e(A)
  & = \int_A \big( 1 - e^{-d(v,g)} \big) \, \frac{\nu^v_\e(dg)}{K^v_\e} \\
  & = \int_{\psi^{-1}(A)} \big( 1 - e^{-d(v,\psi(g))} \big) \, \frac{\ttnu^{j,v}_\e(dg)}{\ttK^v_\e} \\
  & = \ttmu^{j,v}_\e \big( (\psi^j)^{-1} (A) \big).
 \end{aligned}
\end{equation}

Next, introduce the topological subspaces 
    $\overbar{\tcG^j \bs \{v\}}$ of~$\overbar{\ttcG^j \bs \{v\}}$
and $\overbar{\cG^j \bs \{v\}}$ of~$\overbar{\cG \bs \{v\}}$, and
consider the continuous extension 
  of $\psi^j \colon \tcG^j \rightarrow \cG^j$ to $\overbar{\psi}^j \colon \overbar{\tcG^j \bs \{v\}} \rightarrow \overbar{\cG \bs \{v\}}$.
Continuity of $\overbar{\psi}^j$ dictates that the new points $\overbar{\tcG^j \bs \{v\}} \bs \tcG^j$ are mapped to
 \begin{equation} \label{eq:G_GC:psi continuation to compact i}
 \begin{aligned}
  i \in \cI^j_s(v): \quad
    & \overbar{\psi}^j \big( (e^j_i, 0+) \big) = \lim_{x \downdownarrows 0} \psi^j \big( (e^j_i,x) \big) 
    = (i, 0+), \\
    & \overbar{\psi}^j \big( (e^j_i, \cR_i-) \big) = \lim_{x \upuparrows \cR(i)} \psi^j \big( (e^j_i,x) \big) 
    = (i, \cR_i), \\
  i \in \cI^{-j}_s(v):  \quad
    & \overbar{\psi}^j \big( (e^j_i, 0+) \big) = \lim_{x \downdownarrows 0} \psi^j \big( (e^j_i,x) \big) 
    = (i, \cR_i), \\
    & \overbar{\psi}^j \big( (e^j_i, \cR_i-) \big) = \lim_{x \upuparrows \cR(i)} \psi^j \big( (e^j_i,x) \big) 
    = (i, 0+),  
 \end{aligned}
 \end{equation}
and analogously
 \begin{equation} \label{eq:G_GC:psi continuation to compact ii}
 \begin{aligned}
  i \in \cI^j(v): \quad
   & \overbar{\psi}^j \big( (i, 0+) \big) = (i, 0+), \hspace*{0.6em} \text{ if $v = \cLV_-(i)$,} \\ 
   & \overbar{\psi}^j \big( (i, \cR_i-) \big) = (i, \cR_i-), \text{ if $v = \cLV_+(i)$,} \\
  e \in \cE^j(v): \quad
   & \overbar{\psi}^j \big( (e,0+) \big) = (e,0+), \\
  e \in \cE^j: \quad
   & \overbar{\psi}^j \big( (e,+\infty) \big) = (e,+\infty).
 \end{aligned}
 \end{equation}
 
Proceeding in the course of the proof of~\cite[Theorem~1.2]{WernerMetricA} for $\ttX^j$,
we extend the measures $\ttmu^{j,v}_\e$ to measures $\ttomu^{j,v}_\e$ on $\overbar{\ttcG^j \bs \{v\}}$ by
 \begin{align*}
  \ttomu^{j,v}_\e (A) := \ttmu^v_\e \big( A \cap \big(\ttcG^j \bs \{v\} \big) \big), \quad A \in \sB \big( \overbar{\ttcG^j \bs \{v\}} \big),
 \end{align*}
and choose a sequence of positive numbers $(\e_n, n \in \N)$ converging to zero, such that ${\big( \ttomu^{j,v}_{\e_n}, n \in \N \big)}$ converges weakly to a measure $\ttomu^{j,v}$.
When also extending the measures $\mu^v_\e$ to measures $\overbar{\mu}^v_\e$ on $\overbar{\cG \bs \{v\}}$, we obtain with equation~\eqref{eq:G_GC:mu relation}
 \begin{align*}
 \forall A \in \sB \big( \overbar{\cG \bs \{v\}} \big): \quad 
  \overbar{\mu}^v_\e(A) 
  &  = \mu^v_\e \big(A \cap \big(\cG \bs \{v\}\big) \big) \\
  &  = \ttmu^{j,v}_\e \big( (\psi^j)^{-1} \big(A \cap \big(\cG \bs \{v\}\big) \big) \\
  &  = \ttmu^{j,v}_\e \big( (\overbar{\psi}^j)^{-1} (A) \cap \big( \tcG^j \bs \{v\} \big) \big) \\
  &  = \ttmu^{j,v}_\e \big( (\overbar{\psi}^j)^{-1} (A) \cap \big( \ttcG^j \bs \{v\} \big) \big) \\
  &  = \ttomu^{j,v}_\e \circ (\overbar{\psi}^j)^{-1} (A).
 \end{align*}
By the continuous mapping theorem, $(\overbar{\mu}^v_{\e_n}, n \in \N)$ converges weakly to the measure 
 \begin{align*}
  \overbar{\mu}^v = \ttomu^{j,v} \circ (\overbar{\psi}^j)^{-1}
 \end{align*}
on $\overbar{\cG \bs \{v\}}$.
We summarize all of our results up to this point:

\begin{lemma} \label{lem:G_GC:glueing process Feller data}
 Let $v \in \cV^j$, and $K^v_\e$, $\mu^v_\e$, $\overbar{\mu}^v$ and $\ttK^{j,v}_\e$, $\ttmu^{j,v}$, $\ttomu^{j,v}$ 
 be defined as in~\cite[Theorem~1.2]{WernerMetricA} for the Brownian motions $X$, $\ttX^j$ respectively.
 Then,
 \begin{enumerate}
  \item $K^v_\e = \ttK^{j,v}_\e$ for all $\e < \d$,
  \item $\mu^v_\e = \ttmu^{j,v}_\e \circ (\psi^j)^{-1}$ for all $\e < \d$,
  \item $(\overbar{\mu}^v_{\e_n}, n \in \N)$ converges weakly along the same sequence ${(\e_n, n \in \N)}$ of
         positive numbers for which $(\ttomu^{j,v}_{\e_n}, n \in \N)$ converges weakly to $\ttomu^{j,v}$,
         and the limit of $(\overbar{\mu}^v_{\e_n}, n \in \N)$ is
          \begin{align*}
           \overbar{\mu}^v = \ttomu^{j,v} \circ (\overbar{\psi}^j)^{-1}.
          \end{align*}
 \end{enumerate}
\end{lemma}

We are now ready to compute the \FW\ data of the glued process $X$, thus completing the proof of Theorem~\ref{theo:G_GC:glueing process}:

\begin{proof}[Proof of Theorem~\ref{theo:G_GC:glueing process}]
 We have already proved in Lemma~\ref{lem:G_GC:glued process is BM} that $X$ is a Brownian motion on $\cG$.
 It remains to compute the \FW\ data of $X$ by employing Lemma~\ref{lem:G_GC:glueing process Feller data}.
 To this end, let $v \in \cV$ and choose $j \in \{-1, +1\}$ such that~$v \in \cV^j$.
 
 The killing parameters are given by
  \begin{align*}
   c^{v,\D}_1 & = \lim_{n \rightarrow \infty} \frac{\PV_v(X_{\t_{\e_n}} = \D)}{\EV_v(\t_{\e_n}) K^v_{e_n}} 
               = \lim_{n \rightarrow \infty} \frac{\PV^j_v(\ttX^j_{\ttt^j_{\e_n}} = \D)}{\EV^j_v(\ttt^j_{\e_n}) \ttK^{j,v}_{e_n}} 
               = p^{v,\D}_1, \\
   c^{v,\infty}_1 & = \sum_{e \in \cE} \overbar{\mu}^v \big( \{ (e, +\infty) \} \big) 
                   = \sum_{e \in \cE^j \cup \cE^j_s} \ttomu^{j,v} \big( \{ (e, +\infty) \} \big) 
                   = p^{v,\infty}_1,           
  \end{align*}
 and thus vanish, as $p^v_1 = p^{v,\D}_1 + p^{v,\infty}_1 = 0$ holds by assumption.
 
 The reflection parameters are defined as
   \begin{align*}
    c^{v,l}_2 
    & = 
    \begin{cases}
     \overbar{\mu}^v \big( \{ (l, 0+) \} \big), & l \in \cE(v), \\
     \overbar{\mu}^v \big( \{ (l, 0+) \} \big), & l \in \cI(v), v = \cLV_-(l), \\
     \overbar{\mu}^v \big( \{ (l, \cR_l-) \} \big), & l \in \cI(v), v = \cLV_+(l).
    \end{cases} 
  \end{align*}
  For $e \in \cE(v)$, the relation $(\overbar{\psi}^j)^{-1} \big( (e,0+) \big) = (e,0+)$ immediately yields $c^{v,e}_2 = p^{v,e}_2$.
  For $i \in \cI(v)$, we need to distinguish some cases, using equations~\eqref{eq:G_GC:psi continuation to compact i} and~\eqref{eq:G_GC:psi continuation to compact ii}:
  For $i \in \cI(v)$ with $v = \cLV_-(i)$, that is if $i \in \cI^j(v) \cup \cI^j_s(v)$, we have 
  \begin{align*}
    c^{v,i}_2 
      = \overbar{\mu}^v \big( \{ (i, 0+) \} \big)
      = 
       \begin{cases}
        \ttomu^{j,v} \big( \{ (i, 0+) \} \big) = p^{v,i}_2,     & i \in \cI^j(v), \\
        \ttomu^{j,v} \big( \{ (e^j_i, 0+) \} \big) = p^{v,e^j_i}_2, & i \in \cI^j_s(v), 
       \end{cases} 
  \end{align*}
  while for $i \in \cI(v)$ with $v = \cLV_+(i)$, that is if $i \in \cI^{-j}(v) \cup \cI^{-j}_s(v)$, we have
  \begin{align*}
    c^{v,i}_2 
      = \overbar{\mu}^v \big( \{ (i, \cR_i-) \} \big)
      = 
       \begin{cases}
        \ttomu^{j,v} \big( \{ (i, \cR_i-) \} \big) = p^{v,i}_2,  & i \in \cI^{-j}(v), \\
        \ttomu^{j,v} \big( \{ (e^j_i, 0+) \} \big) = p^{v,e^j_i}_2,   & i \in \cI^{-j}_s(v).
       \end{cases} 
  \end{align*}
  
  The diffusion parameter is given by
   \begin{align*}
    c^v_3 
    = \lim_{n \rightarrow \infty} \frac{1}{K^v_{e_n}} 
    = \lim_{n \rightarrow \infty} \frac{1}{\ttK^{j,v}_{e_n}} 
    = p^v_3.
   \end{align*}

  For all $A \in \sB(\cG \bs \{v\})$, the jump distribution is computed by
    \begin{align*}
       c^v_4(A)
       & = \int_A \frac{1}{1-e^{-d(v,g)}} \overbar{\mu}^v(dg) \\
       & = \int_{(\psi^j)^{-1}(A)} \frac{1}{1-e^{-d(v,\psi^j(g))}} \ttomu^{j,v}(dg) \\
       & = p^v_4 \circ (\psi^j)^{-1}(A),
    \end{align*}
  as $\overbar{\psi^j}$ is an extension from $\psi^j \colon \tcG^j \rightarrow \cG$ and an isometry.        
\end{proof}

\section{Completing the Construction} \label{sec:G_GC:completing}

We are ready to carry out the construction that was laid out in section~\ref{sec:G_GC:agenda}.

\begin{theorem} \label{theo:G_GC:complete glueing, infinite lifetime}
 Let $\cG = (\cV, \cE, \cI, \cLV, \cR)$ be a metric graph, and
 for every $v \in \cV$ let constants
 $p^{v,l}_2 \geq 0$ for each $l \in \cL(v)$, $p^v_3 \geq 0$ and a measure $p^v_4$ on $\cG \bs \{v\}$ be given, satisfying
  \begin{align*}
    \sum_{l \in \cL(v)} p^{v,l}_2 + p^v_3 + \int_{\cG \bs \{v\}} \big( 1 - e^{-d(v,g)} \big) \, p^v_4 (dg) = 1,
  \end{align*}  
  and
  \begin{align*}
    p^v_4 \big( \cG \bs \{v\} \big) = +\infty, \quad \text{if} \quad \sum_{l \in \cL(v)} p^{v,l}_2 + p^v_3 = 0,
  \end{align*}
  as well as $p^v_4 \big( \comp \overline{\BB_\d(v)} \big) = 0$ for some $\d \in ( 0, \min_{l \in \cL} \cR_l )$.
  Then there exists a Brownian motion~$X$ on $\cG$ which has infinite lifetime, is continuous inside all edges,
  satisfies ${X_{\t_\e} \in \BB_\d(v)}$ $\PV_v$-a.s.\ for all $\e < \d$, $v \in \cV$, and admits the \FW\ data
   \begin{align*}
    \big( 0, (p^{v,l}_2)_{l \in \cL(v)}, p^v_3, p^v_4 \big)_{v \in \cV}.
   \end{align*}  
\end{theorem}
\begin{proof}
 We proceed via an induction over the count $n := \abs{\cV}$ of vertices.
 If $n = 1$, then $\cG$ is a star graph, so the construction
 given in~\cite{WernerStar} together with \cite[Theorem~4.33]{WernerStar}, \cite[Lemma~4.1]{WernerMetricA}
 and \cite[Theorem~4.3]{WernerStar}
 yield the result. 
 
 Assume now that such Brownian motions exist for all metric graphs with less than $n$ vertices. 
 Let $\cG$ be a metric graph with $n$ vertices $\cV = \{ v_1, \ldots, v_n \}$ and boundary data as given in the theorem.
 We decompose the graph into~$\ttcG^{-1}$ and~$\ttcG^{+1}$, as done in section~\ref{sec:G_GC:glueing}, for $\cV^{-1} = \{ v_1, \ldots, v_{n-1} \}$ and $\cV^{+1} = \{ v_n \}$. 
 Then the conditions of the theorem are satisfied for these graphs $\ttcG^{-1}$, $\ttcG^{+1}$ with $n-1$ vertices, one vertex respectively,
 and corresponding boundary data $(p^{v,l}_2 \geq 0, l \in \ttcL^j(v))$, $p^v_3 \geq 0$ and $p^v_4 \circ \psi^j$ (as $\psi^j$ is an isometry, this data satisfies the normalization requirements).
 Therefore, there exist
 Brownian motions $\ttX^{j}$ on $\ttcG^{j}$ with infinite lifetime which are continuous inside all edges, 
 satisfy $\ttX^j_{\ttt^j_\e} \in \BB_\d(v)$ $\PV^j_v$-a.s.\ for all $v \in \cV^j$ and admit the \FW\ data 
  \begin{align*}
   \big( 0, (p^{v,l}_2)_{l \in \ttcL^j(v)}, p^v_3, p^v_4 \circ \psi^j \big)_{v \in \cV^{j}}
  \end{align*}
 with $p^{v, e^j_i} := p_2^{v,i}$ for $i \in \cI_s(v)$, $v \in \cV^j$.
 We then follow the construction of section~\ref{sec:G_GC:glueing} in order to glue $\ttX^{-1}$ and $\ttX^{+1}$ 
 together, and Theorem~\ref{theo:G_GC:glueing process} concludes the proof.
\end{proof} 

In order to implement the killing parameter and the non-local jumps, we first need to adjoin the ``fake cemeteries'' $\sq^v$ for all $v \in \cV$:

\begin{theorem}  \label{theo:G_GC:complete glueing, cemetries}
 Let $\cG = (\cV, \cE, \cI, \cLV, \cR)$ be a metric graph, and
 for every $v \in \cV$ let constants
 $p^v_1 \geq 0$, $p^{v,l}_2 \geq 0$ for each $l \in \cL(v)$, $p^v_3 \geq 0$ and a measure $p^v_4$ on $\cG \bs \{v\}$ be given with
  \begin{align*}
    p^v_1 + \sum_{l \in \cL(v)} p^{v,l}_2 + p^v_3 + \int_{\cG \bs \{v\}} \big( 1 - e^{-d(v,g)} \big) \, p^v_4 (dg) = 1,
  \end{align*}  
  and
  \begin{align*}
    p^v_4 \big( \cG \bs \{v\} \big) = +\infty, \quad \text{if} \quad \sum_{l \in \cL(v)} p^{v,l}_2 + p^v_3 = 0,
  \end{align*}
  as well as $p^v_4 \big( \comp \overline{\BB_\d(v)} \big) = 0$ for some $\d \in \big( 0, \min_{l \in \cL} \cR_l \big)$.
  Then there exists a Brownian motion $X$ on $\cG \cup \{ \sq^v, v \in \cV \}$ with $\{ \sq^v, v \in \cV \}$ being an isolated, absorbing set for $X$,
  such that $X$ has infinite lifetime, is continuous inside all edges, satisfies
  $X_{\t_\e} \in \BB_\d(v) \cup \{\sq^v\}$ $\PV_v$-a.s.\ for all $\e < \d$, $v \in \cV$, and has the \FW\ data
   \begin{align*}
    \big( 0, (p^{v,l}_2)_{l \in \cL(v)}, p^v_3, p^v_4 + p^v_1 \, \e_{\sq^v} \big)_{v \in \cV}.
   \end{align*}  
\end{theorem}
\begin{proof}
 This proof proceeds analogously to the proof of Theorem~\ref{theo:G_GC:complete glueing, infinite lifetime}, except that we need to adjoin the
 isolated points $\sq^v$, $v \in \cV$, to the partial processes and revive these processes there before gluing the partial graphs together.

 If $\abs{\cV} = 1$, then $\cG$ is a star graph, and the construction of~\cite{WernerStar} (again with \cite[Theorem~4.33]{WernerStar},
 \cite[Lemma~4.1]{WernerMetricA}, 
 and \cite[Theorem~4.3]{WernerStar}) gives a Brownian motion on~$\cG$ with the needed properties and \FW\ data   
   \begin{align*}
    \big( p^v_1, (p^{v,l}_2)_{l \in \cL(v)}, p^v_3, p^v_4 \big).
   \end{align*}
 By concatenating it with the constant process on $\{\sq^v\}$ with the technique of~\cite{WernerConcat},
 we revive this Brownian motion on a new, isolated, absorbing point $\sq^v$.
 Then a computation along the lines of Lemma~\ref{lem:G_GC:reviving BB with revival kernel} yields that the revived process is a Brownian 
 motion on $\cG \cup \{ \sq^v \}$ with its \FW\ data at $v$ being given by
   \begin{align*}
    \big( 0, (p^{v,l}_2)_{l \in \cL(v)}, p^v_3, p^v_4 + p^v_1 \, \e_{\sq^v} \big).
   \end{align*} 
   
 Now let $\cV = \{ v_1, \ldots, v_n \}$, and assume that the assertion of the theorem holds for any graph with less than $n$ vertices.
 We decompose the graph $\cG$ into $\ttcG^{-1}$ and $\ttcG^{+1}$, as done in section~\ref{sec:G_GC:glueing}, for $\cV^{-1} = \{ v_1, \ldots, v_{n-1} \}$ and $\cV^{+1} = \{ v_n \}$.
 By assumption, there exist Brownian motions $\ttX^j$ on $\ttcG^j \cup \big\{ \sq^v, v \in \cV^j \big\}$ with the needed path properties and \FW\ data 
  \begin{align*}
   \big( 0, (p^{v,l}_2)_{l \in \ttcL^j(v)}, p^v_3, p^v_4 \circ \psi^j + p^v_1 \, \e_{\sq^v} \big)_{v \in \cV^{j}}
  \end{align*} 
 with $p^{v, e^j_i} := p_2^{v,i}$ for $i \in \cI_s(v)$, $v \in \cV^j$.
 We then again follow the construction of section~\ref{sec:G_GC:glueing} to glue $\ttX^{-1}$ and $\ttX^{+1}$ together,
 and Theorem~\ref{theo:G_GC:glueing process} yields the result.
\end{proof} 

In order to complete the proof of the existence theorem for Brownian motions on metric graphs with non-local boundary conditions,
it remains to implement the ``global'' jumps:

\begin{proof}[Proof of Theorem~\ref{theo:G_GC:complete construction}]
 Let $\d > 0$ with $\d < \min_{l \in \cL} \cR_l$, and define for every $v \in \cV$
  \begin{align*}
   q^v_1 := p^v_1 + p_4^v \big( \comp \BB_v(\d) \big), \quad q^v_4 := \restr{p_4^v}{\BB_v(\d)}.
  \end{align*}
 The introduction of the normalizing factor
  \begin{align*}
   c_0^v := \Big( q^v_1 + \sum_{l \in \cL(v)} p^{v,l}_2 + p^v_3 + \int_{\cG \bs \{v\}} \big( 1 - e^{-d(v,g)} \big) \, q^v_4 (dg) \Big)^{-1}
  \end{align*}
 enables us to employ Theorem~\ref{theo:G_GC:complete glueing, cemetries} in order to construct a Brownian motion~$X^1$ on~$\cG \cup \{ \sq^v, v \in \cV \}$
 which has infinite lifetime, is continuous inside all edges,
 satisfies ${X_{\t_\e} \in \BB_\d(v) \cup \{\sq^v\}}$ $\PV_v$-a.s.\ for all $\e < \d$, $v \in \cV$, and has the \FW\ data
   \begin{align*}
     \big( c_0^v \, \big( 0, (p^{v,l}_2)_{l \in \cL(v)}, p^v_3, q^v_4 + q^v_1 \, \e_{\sq^v} \big) \big)_{v \in \cV}.
   \end{align*}  
   
 As $X^1$ has infinite lifetime, we can use an application of the concatenation techniques (see~\cite{WernerConcat} and~\cite[Proposition~14.20]{Sharpe88})
 to adjoin a new, isolated, absorbing point $\sq$ to $X^1$,
 resulting in a Brownian motion $X^2$ on the metric graph $\cG \cup \{ \sq^v, v \in \cV \} \cup \{ \sq \}$ with the same \FW\ data as $X^1$ for all $v \in \cV$, 
 and additional \FW\ data $(0,0,1,0)$ at the new vertex $\sq$.
 
 Let $X^3$ be the right process on $\cG \cup \{ \sq \}$ which results from killing $X^2$ on the absorbing set $\{ \sq^v, v \in \cV \}$ 
 (see Appendix~\ref{app:C_MS:killing on absorbing set}).
 As $X^3$ is strongly Markovian and $X^3_t = X^2_t$ for all $t \leq \h_\cV$, $X^3$ is a Brownian motion on $\cG \cup \{ \sq \}$, and Lemma~\ref{lem:G_GC:killing on absorbing set, Feller data} 
 asserts that the \FW\ data of $X^3$ reads 
   \begin{align*}
     \big( c_0^v \, \big( q^v_1, (p^{v,l}_2)_{l \in \cL(v)}, p^v_3, q^v_4 \big) \big)_{v \in \cV}.
   \end{align*}  
   
 Now construct $X^4$ as the revived process obtained from $X^3$ by the identical copies method with revival distributions 
  \begin{align*}
   \kappa^v := (q^v_1)^{-1} \, \big( p^v_1 \, \e_{\sq} + \restr{p_4^v}{\comp \BB_\d(v)} \big), \quad v \in \cV.
  \end{align*}
 Then by Lemma~\ref{lem:G_GC:reviving BB with revival kernel}, $X^4$ is a Brownian motion on $\cG \cup \{ \sq \}$, and its generator satisfies
   \begin{align*}
     \sD(A)  \subseteq
       \Big\{ & f \in \cC^2_0(\cG \cup \{ \sq \}): \forall v \in \cV: \\
          & - \sum_{l \in \cL(v)} c_0^v \, p^{v,l}_2 \, f_l'(v) + \frac{c_0^v \, p^v_3}{2} f''(v) \\
          & - \int_{(\cG \bs \{v\}) \cup \{ \sq \}} \big( f(g) - f(v) \big) \, c_0^v \big( \restr{p_4^v}{\BB_\d(v)} + p^v_1 \, \e_\sq + \restr{p_4^v}{\BB_\d(v)^\comp} \big) (dg) = 0 \Big\} \\
     = \Big\{ & f \in \cC^2_0(\cG \cup \{ \sq \}): \forall v \in \cV: \\
          & - \sum_{l \in \cL(v)} p^{v,l}_2 \, f_l'(v) + \frac{p^v_3}{2} f''(v)  \\
          & - \int_{(\cG \bs \{v\}) \cup \{ \sq \}} \big( f(g) - f(v) \big) \,  \big( p_4^v + p^v_1 \, \e_\sq \big) (dg) = 0 \Big\}.
   \end{align*}
   
 Finally, employ once more the transformation of Appendix~\ref{app:C_MS:killing on absorbing set} in order to kill~$X^4$ on the isolated, absorbing set~$\{ \sq \}$ 
 and obtain the Brownian motion~$X^5$ on~$\cG$. Lemma~\ref{lem:G_GC:killing on absorbing set, generator data} asserts that the domain of its generator satisfies 
    \begin{align*}
     \sD(A) \subseteq
       \Big\{ & f \in \cC^2_0(\cG): \forall v \in \cV: \\
          & p^v_1 f(v) - \sum_{l \in \cL(v)} p^{v,l}_2 \, f_l'(v) + \frac{p^v_3}{2} f''(v) - \int_{\cG \bs \{v\}} \big( f(g) - f(v) \big) \, p^v_4 (dg) = 0  \Big\}.
   \end{align*}
\end{proof}

\begin{appendix}
 
\section{Killing on an Absorbing Set} \label{app:C_MS:killing on absorbing set}

We present an easy technique to kill a right process on an absorbing set (see \cite[Definition~12.27]{Sharpe88}), 
which will be used in the main construction of this article. 
For the role of the cemetery point~$\D$ and the lifetime conventions in the context of right processes, the reader may consult~\cite[Section~11]{Sharpe88}.

Let $\tE = E_\D \uplus F$ be the topological union of two disjoint Radon spaces $E_\D$ and~$F$, and consider a right process $X$ on $\tE$, with $F$ being an absorbing set for $X$.
We kill the process~$X$ on this absorbing set~$F$ by mapping $F$ to $\D$ with
 \begin{align*}
  \psi \colon \tE \rightarrow E_\D, ~ x \mapsto \psi(x) :=
    \begin{cases}
     x,  & x \in E_\D, \\
     \D, & x \in F.
    \end{cases}
 \end{align*}

By checking the consistency conditions for state space transformations of right processes (see \cite[Section~13]{Sharpe88} and \cite[Section~3.1]{WernerConcat}), we show:
 
\begin{theorem} \label{theo:C_MS:killing on absorbing set}
 $\psi(X)$ is a right process on $E_\D$.
\end{theorem}
\begin{proof}
 The transformation $\psi$ is clearly surjective and measurable, as 
  \begin{align*}
   \forall B \in \sE_\D^u: \quad
   \psi^{-1}(B) = \begin{cases}
                    B,        & \D \notin B, \\
                    B \cup F, & \D \in B.
                  \end{cases}
  \end{align*}
 Let $H_F$ be the first entry time of~$X$ into~$F$.
 We have $X_t \in F$ for all $t \geq \h_F$ a.s., 
 as the strong Markov property at $\h_F$ yields
  \begin{align*}
   \PV \big( X_{\h_F + t} \in F \text{ for all $t \geq 0$} \big)
   & = \EV \big( \PV_{X_{\h_F}} ( X_t \in F \text{ for all $t \geq 0$} ) \big)
     = 1.
  \end{align*}
 Furthermore, it is evident that $X_t \notin F$ for all $t < \h_F$, so the transformed process
  \begin{align*}
    t \mapsto \psi(X_t) = 
     \begin{cases}
       X_t, & t < \h_F, \\
       \D,  & t \geq \h_F 
     \end{cases}
  \end{align*}
 is a.s.\ right continuous.
 
 For all $f \in b\sE_\D^u$, $x \in \tE$, we have for the semigroup $(T_t, t \geq 0)$ of $X$:
  \begin{align*}
   T_t (f \circ \psi) (x)
   & = \EV_x \big( f \circ \psi(X_t) \big) \\
   & = \EV_x \big( f(X_t) \,;\, t < \h_F \big) + f(\D) \, \PV_x \big( t \geq  \h_F \big) \\
   & = g \circ \psi(x),
  \end{align*}
 with $g \in b\sE_\D^u$ being defined by
  \begin{align*}
   g(x) :=
    \begin{cases}
     \EV_x \big( f(X_t) \,;\, t < \h_F  \big) + f(\D) \, \PV_x \big( t \geq  \h_F \big), & x \in E, \\
     f(\D), & x = \D,
    \end{cases}
  \end{align*}
 as $\h_F = 0$ holds $\PV_x$-a.s.\ for all $x \in F$. 
\end{proof}

\end{appendix}

%% file: G_GC_glueing_1.tex
\begin{tikzpicture}[scale=0.45]
  \node (G1) at (0,6) {$\ttcG^{-1}$};
  \node[blue] (G2) at (13,6) {$\ttcG^{+1}$};

  \tikzstyle{every node}=[draw,shape=circle, fill, style={transform shape}];

  \node (A) at (2,5+0.25) { };
  \node (B) at (0, 0+0.25) { };
  \node (C) at (4, 0+0.25) { };
  \node[blue] (D) at (10, 0) { };
  \node[blue] (E) at (13, 3) { };
  \node[blue] (F) at (10, 5) { }; 
  
  \tikzstyle{every node}=[draw=none];
  
  \node (sA) at ($(A) + (0,-0.25)$) { };
  \node (sC) at ($(C) + (0,-0.25)$) { };
  \node (sD) at ($(D) + (0,0.25)$) { };
  \node (sF) at ($(F) + (0,0.25)$) { };

  \coordinate (A_1) at ($ (A) + (0, 1.4)$);
  \coordinate (B_1) at ($ (B) + (-1, 1)$);
  \coordinate (B_2) at ($ (B) + (0, -1.4)$);
  
  \coordinate (D_1) at ($ (D) + (1, -1)$);
  \coordinate (D_2) at ($ (D) + (-1, -1)$);
  \coordinate (E_1) at ($ (E) + (1, 1)$);
  \coordinate (E_2) at ($ (E) + (1, -1)$); 
  \coordinate (F_1) at ($ (F) + (0, 1.4)$); 
  
  \node[left] (tA) at (A) {$v_1$};
  \node[below left] (tB) at (B) {$v_2$};
  \node[below] (tC) at (C) {$v_3$};
  \node[below, blue] (tD) at (D) {$v_4$};
  \node[above left, blue] (tE) at (E) {$v_5$};
  \node[right, blue] (tF) at (F) {$v_6$}; 
  
  \draw[-] (A) -- (B);
  \draw[-] (B) -- (C);
  \draw[-] (C) -- (A);
  
  \draw[blue] (D) -- (E);
  \draw[blue] (D) .. controls (12, 1) .. (E);  
  \draw[blue] (D) .. controls (11, 2) .. (E);  
  \draw[blue] (D) -- (F);

  \draw[-] (A) -- (A_1);
  \draw[densely dotted] (A_1) -- ($1.3*(A_1) - 0.3*(A)$);
  \draw[-] (B) -- (B_1);
  \draw[densely dotted] (B_1) -- ($1.3*(B_1) - 0.3*(B)$);
  \draw[-] (B) -- (B_2);
  \draw[densely dotted] (B_2) -- ($1.3*(B_2) - 0.3*(B)$);

  \draw[blue] (D) -- (D_1);
  \draw[blue, densely dotted] (D_1) -- ($1.3*(D_1) - 0.3*(D)$);
  \draw[blue] (D) -- (D_2);
  \draw[blue, densely dotted] (D_2) -- ($1.3*(D_2) - 0.3*(D)$);
  \draw[blue] (E) -- (E_1);
  \draw[blue, densely dotted] (E_1) -- ($1.3*(E_1) - 0.3*(E)$);
  \draw[blue] (E) -- (E_2);  
  \draw[blue, densely dotted] (E_2) -- ($1.3*(E_2) - 0.3*(E)$);
  \draw[blue] (F) -- (F_1);
  \draw[blue, densely dotted] (F_1) -- ($1.3*(F_1) - 0.3*(F)$);
  
  \draw[-] (C) -- ($0.8*(sD)+0.2*(C)$);
  \draw[densely dashed] ($0.8*(sD)+0.2*(C)$) -- ($0.95*(sD)+0.05*(C)$);
  \draw[blue] (D) -- ($0.8*(sC)+0.2*(D)$);
  \draw[blue, densely dashed] ($0.8*(sC)+0.2*(D)$) -- ($0.95*(sC)+0.05*(D)$);
  
  \draw[-] (C) -- ($0.8*(sF)+0.2*(C)$);
  \draw[densely dashed] ($0.8*(sF)+0.2*(C)$) -- ($0.90*(sF)+0.10*(C)$);
  \draw[blue] (F) -- ($0.8*(sC)+0.2*(F)$);
  \draw[blue, densely dashed] ($0.8*(sC)+0.2*(F)$) -- ($0.90*(sC)+0.10*(F)$);
  
 
  \draw[blue] (F) .. controls ($0.5*(sF)+0.5*(A) + (-1,0.3)$) .. ($(A) - (-1.5,-0.20)$);  
  \draw[blue, densely dashed] ($(A) - (-1.5,-0.20)$) -- ($(A) - (-0.5,-0.05)$);
  \draw[-] (A) .. controls ($0.5*(sF)+0.5*(A) + (-2,0.5)$) .. ($(sF) - (1.5,-0.10)$);  
  \draw[densely dashed] ($(sF) - (1.5,-0.10)$) -- ($(sF) - (0.5,0.05)$); 
 
  \draw[-] (A) .. controls ($0.5*(sA)+0.5*(F) - (-1,0.3)$) .. ($(F) + (-1.5,-0.20)$);  
  \draw[densely dashed] ($(F) + (-1.5,-0.20)$) -- ($(F) + (-0.5,-0.05)$);
  \draw[blue] (F) .. controls ($0.5*(sA)+0.5*(F) - (-2,0.5)$) .. ($(sA) + (1.5,-0.10)$);  
  \draw[blue, densely dashed] ($(sA) + (1.5,-0.10)$) -- ($(sA) + (0.5,-0.00)$); 
  
  \node[blue] (sFA1) at ($0.5*(F)+0.5*(A)+(1.0,0.9)$) {$e_{i_4}^{-1}$};
  \node (sAF1) at ($0.5*(F)+0.5*(A)+(-1.0,1.1)$) {$e_{i_4}^{+1}$};
  \node[blue] (sFA2) at ($0.5*(F)+0.5*(A)+(1.0,-1.1)$) {$e_{i_5}^{-1}$};
  \node (sAF2) at ($0.5*(F)+0.5*(A)+(-1.0,-0.9)$) {$e_{i_5}^{+1}$};
  \node[blue] (sFC) at ($0.5*(F)+0.5*(C)+(1.0,0)$) {$e_{i_6}^{-1}$};
  \node (sCF) at ($0.5*(F)+0.5*(C)+(-1.0,0)$) {$e_{i_6}^{+1}$};
  \node[blue] (sDC) at ($0.5*(D)+0.5*(C)+(0.0,-0.7)$) {$e_{i_7}^{-1}$};
  \node (sCD) at ($0.5*(D)+0.5*(C)+(0.0,+0.7)$) {$e_{i_7}^{+1}$};
  
\end{tikzpicture}

%% file: G_GC_glueing_2.tex
\begin{tikzpicture}[scale=0.45]
  \node (G1) at (0,6) {$\cG^{-1}$};
  \node[blue] (G2) at (13,6) {$\cG^{+1}$};

  \tikzstyle{every node}=[draw,shape=circle, fill, style={transform shape}];

  \node (A) at (2,5+0.25) { };
  \node (B) at (0, 0+0.25) { };
  \node (C) at (4, 0+0.25) { };
  \node[blue] (D) at (10, 0) { };
  \node[blue] (E) at (13, 3) { };
  \node[blue] (F) at (10, 5) { }; 
  
  \tikzstyle{every node}=[draw=none];
  
  \node (sA) at ($(A) + (0,-0.25)$) { };
  \node (sC) at ($(C) + (0,-0.25)$) { };
  \node (sD) at ($(D) + (0,0.25)$) { };
  \node (sF) at ($(F) + (0,0.25)$) { };

  \coordinate (A_1) at ($ (A) + (0, 1.4)$);
  \coordinate (B_1) at ($ (B) + (-1, 1)$);
  \coordinate (B_2) at ($ (B) + (0, -1.4)$);
  
  \coordinate (D_1) at ($ (D) + (1, -1)$);
  \coordinate (D_2) at ($ (D) + (-1, -1)$);
  \coordinate (E_1) at ($ (E) + (1, 1)$);
  \coordinate (E_2) at ($ (E) + (1, -1)$); 
  \coordinate (F_1) at ($ (F) + (0, 1.4)$);

  \node[left] (tA) at (A) {$v_1$};
  \node[below left] (tB) at (B) {$v_2$};
  \node[below] (tC) at (C) {$v_3$};
  \node[below, blue] (tD) at (D) {$v_4$};
  \node[above left, blue] (tE) at (E) {$v_5$};
  \node[right, blue] (tF) at (F) {$v_6$};

  \draw[-] (A) -- (B);
  \draw[-] (B) -- (C);
  \draw[-] (C) -- (A);
  
  \draw[blue] (D) -- (E);
  \draw[blue] (D) .. controls (12, 1) .. (E);  
  \draw[blue] (D) .. controls (11, 2) .. (E);  
  \draw[blue] (D) -- (F);

  \draw[-] (A) -- (A_1);
  \draw[densely dotted] (A_1) -- ($1.3*(A_1) - 0.3*(A)$);
  \draw[-] (B) -- (B_1);
  \draw[densely dotted] (B_1) -- ($1.3*(B_1) - 0.3*(B)$);
  \draw[-] (B) -- (B_2);
  \draw[densely dotted] (B_2) -- ($1.3*(B_2) - 0.3*(B)$);

  \draw[blue] (D) -- (D_1);
  \draw[blue, densely dotted] (D_1) -- ($1.3*(D_1) - 0.3*(D)$);
  \draw[blue] (D) -- (D_2);
  \draw[blue, densely dotted] (D_2) -- ($1.3*(D_2) - 0.3*(D)$);
  \draw[blue] (E) -- (E_1);
  \draw[blue, densely dotted] (E_1) -- ($1.3*(E_1) - 0.3*(E)$);
  \draw[blue] (E) -- (E_2);  
  \draw[blue, densely dotted] (E_2) -- ($1.3*(E_2) - 0.3*(E)$);
  \draw[blue] (F) -- (F_1);
  \draw[blue, densely dotted] (F_1) -- ($1.3*(F_1) - 0.3*(F)$);
  
  \tikzstyle{every node}=[draw,shape=circle, style={transform shape}];
  \node[blue] (As) at (A) { $\cdot$ };
  \node[black] (Fs) at (F) { $\cdot$ };
  \node[blue] (Cs) at (C) { $\cdot$ };
  \node[black] (Ds) at (D) { $\cdot$ };
  
  \tikzstyle{every node}=[draw=none];
  
  \draw[-] (C).. controls ($0.9*(Ds)+0.1*(C)+(0,0.1)$) .. ($(Ds) + (-0.22, 0.12)$);
  \draw[blue] (D) .. controls ($0.9*(Cs)+0.1*(D)-(0,0.1)$) .. ($(Cs) - (-0.22, 0.12)$);
  
  \draw[-] (C) .. controls ($0.9*(sF)+0.1*(C)$) .. ($(Fs) + (-0.22, -0.12)$);
  \draw[blue] (F) .. controls ($0.9*(Cs)+0.1*(F)+(0.1,-0.1)$) .. ($(Cs) - (-0.22, -0.12)$);
  
 
  \draw[blue] (F) .. controls ($0.5*(sF)+0.5*(A) + (-1,0.3)$) .. ($(As) - (-1.5,-0.20)$) node [red, midway, above] (sFA1) {$i_4$};  
  \draw[blue] ($(As) - (-1.5,-0.20)$) -- ($(As) + (0.23,0.02)$);
  \draw[-] (A) .. controls ($0.5*(sF)+0.5*(A) + (-2,0.5)$) .. ($(sF) - (1.5,-0.05)$);  
  \draw[-] ($(sF) - (1.5,-0.05)$) -- ($(Fs) + (-0.23,0.12)$); 
 
  \draw[-] (A) .. controls ($0.5*(sA)+0.5*(F) - (-1,0.3)$) .. ($(F) + (-1.5,-0.20)$)  node [red, midway, below] (sFA1) {$i_5$};  
  \draw[-] ($(F) + (-1.5,-0.20)$) -- ($(Fs) + (-0.23,-0.02)$);
  \draw[blue] (F) .. controls ($0.5*(sA)+0.5*(F) - (-2,0.5)$) .. ($(sA) + (1.5,-0.10)$);  
  \draw[blue] ($(sA) + (1.5,-0.10)$) -- ($(As) + (0.20,-0.15)$); 
  
  \node[red] (sCF) at ($0.5*(F)+0.5*(C)+(-1.0,0)$) {$i_6$};
  \node[red] (sDC) at ($0.5*(D)+0.5*(C)+(0.0,-0.5)$) {$i_7$};
  
\end{tikzpicture}

%% file: G_GC_constr_1.tex
\begin{tikzpicture}[scale=0.35]
  \tikzstyle{every node}=[draw,shape=circle, fill, style={transform shape}];

  \node (A) at (2,5+0.25) { };
  \node (B) at (0, 0+0.25) { };
  \node (C) at (4, 0+0.25) { };
  \node[blue] (D) at (10, 0) { };
  \node[blue] (E) at (13, 3) { };
  \node[blue] (F) at (10, 5) { }; 
  
  \tikzstyle{every node}=[draw=none];
  
  \node (sA) at ($(A) + (0,-0.25)$) { };
  \node (sC) at ($(C) + (0,-0.25)$) { };
  \node (sD) at ($(D) + (0,0.25)$) { };
  \node (sF) at ($(F) + (0,0.25)$) { };

  \coordinate (A_1) at ($ (A) + (0, 1.4)$);
  \coordinate (B_1) at ($ (B) + (-1, 1)$);
  \coordinate (B_2) at ($ (B) + (0, -1.4)$);
  
  \coordinate (D_1) at ($ (D) + (1, -1)$);
  \coordinate (D_2) at ($ (D) + (-1, -1)$);
  \coordinate (E_1) at ($ (E) + (1, 1)$);
  \coordinate (E_2) at ($ (E) + (1, -1)$); 
  \coordinate (F_1) at ($ (F) + (0, 1.4)$);

  \draw[-] (A) -- (B);
  \draw[-] (B) -- (C);
  \draw[-] (C) -- (A);
  
  \draw[blue] (D) -- (E);
  \draw[blue] (D) .. controls (12, 1) .. (E);  
  \draw[blue] (D) .. controls (11, 2) .. (E);  
  \draw[blue] (D) -- (F);

  \draw[-] (A) -- (A_1);
  \draw[densely dotted] (A_1) -- ($1.3*(A_1) - 0.3*(A)$);
  \draw[-] (B) -- (B_1);
  \draw[densely dotted] (B_1) -- ($1.3*(B_1) - 0.3*(B)$);
  \draw[-] (B) -- (B_2);
  \draw[densely dotted] (B_2) -- ($1.3*(B_2) - 0.3*(B)$);

  \draw[blue] (D) -- (D_1);
  \draw[blue, densely dotted] (D_1) -- ($1.3*(D_1) - 0.3*(D)$);
  \draw[blue] (D) -- (D_2);
  \draw[blue, densely dotted] (D_2) -- ($1.3*(D_2) - 0.3*(D)$);
  \draw[blue] (E) -- (E_1);
  \draw[blue, densely dotted] (E_1) -- ($1.3*(E_1) - 0.3*(E)$);
  \draw[blue] (E) -- (E_2);  
  \draw[blue, densely dotted] (E_2) -- ($1.3*(E_2) - 0.3*(E)$);
  \draw[blue] (F) -- (F_1);
  \draw[blue, densely dotted] (F_1) -- ($1.3*(F_1) - 0.3*(F)$);
  
  \draw[-] (C) -- ($0.8*(sD)+0.2*(C)$);
  \draw[densely dashed] ($0.8*(sD)+0.2*(C)$) -- ($0.95*(sD)+0.05*(C)$);
  \draw[blue] (D) -- ($0.8*(sC)+0.2*(D)$);
  \draw[blue, densely dashed] ($0.8*(sC)+0.2*(D)$) -- ($0.95*(sC)+0.05*(D)$);
  
  \draw[-] (C) -- ($0.8*(sF)+0.2*(C)$);
  \draw[densely dashed] ($0.8*(sF)+0.2*(C)$) -- ($0.90*(sF)+0.10*(C)$);
  \draw[blue] (F) -- ($0.8*(sC)+0.2*(F)$);
  \draw[blue, densely dashed] ($0.8*(sC)+0.2*(F)$) -- ($0.90*(sC)+0.10*(F)$);
  
 
  \draw[blue] (F) .. controls ($0.5*(sF)+0.5*(A) + (-1,0.3)$) .. ($(A) - (-1.5,-0.20)$);  
  \draw[blue, densely dashed] ($(A) - (-1.5,-0.20)$) -- ($(A) - (-0.5,-0.05)$);
  \draw[-] (A) .. controls ($0.5*(sF)+0.5*(A) + (-2,0.5)$) .. ($(sF) - (1.5,-0.10)$);  
  \draw[densely dashed] ($(sF) - (1.5,-0.10)$) -- ($(sF) - (0.5,0.05)$); 
 
  \draw[-] (A) .. controls ($0.5*(sA)+0.5*(F) - (-1,0.3)$) .. ($(F) + (-1.5,-0.20)$);  
  \draw[densely dashed] ($(F) + (-1.5,-0.20)$) -- ($(F) + (-0.5,-0.05)$);
  \draw[blue] (F) .. controls ($0.5*(sA)+0.5*(F) - (-2,0.5)$) .. ($(sA) + (1.5,-0.10)$);  
  \draw[blue, densely dashed] ($(sA) + (1.5,-0.10)$) -- ($(sA) + (0.5,-0.00)$);

  
  \tikzstyle{every node}=[red, inner sep = 1pt];
  
  \node (dA) at (-0.5,5.5) {$\square^1$};
  \draw[red, ->] (A) .. controls (1,5.5) .. (dA);
  
  \node (dB) at (-1.5,-2) {$\square^2$};
  \draw[red, ->] (B) .. controls (-0.5,-0.5).. (dB);
  
  \node (dC) at (4,-2) {$\square^3$};
  \draw[red, ->] (C) .. controls (4,-0.5).. (dC);
  
  \node (dD) at (13,-2) {$\square^4$};
  \draw[red, ->] (D) .. controls (11.5,-0.5).. (dD);
  
  \node (dE) at (13,1) {$\square^5$};
  \draw[red, ->] (E) -- (dE);
  
  \node (dF) at (12.5,5.5) {$\square^6$};
  \draw[red, ->] (F) .. controls (11,5.5) .. (dF);
  
  	  \draw[densely dotted] (A) ellipse (1 and 1);
	  \draw[densely dotted] (B) ellipse (1 and 1);
	  \draw[densely dotted] (C) ellipse (1 and 1);
	  \draw[blue, densely dotted] (D) ellipse (1 and 1);
	  \draw[blue, densely dotted] (E) ellipse (1 and 1);
	  \draw[blue, densely dotted] (F) ellipse (1 and 1);
\end{tikzpicture}

%% file: G_GC_constr_2.tex
\begin{tikzpicture}[scale=0.35]
  \tikzstyle{every node}=[draw,shape=circle, fill, style={transform shape}];

  \node (A) at (2, 5) { };
  \node (B) at (0, 0) { };
  \node (C) at (4, 0) { };
  \node (D) at (10, 0) { };
  \node (E) at (13, 3) { };
  \node (F) at (10, 5) { }; 
  
  \tikzstyle{every node}=[];

  \coordinate (A_1) at ($ (A) + (0, 1.4)$);
  \coordinate (B_1) at ($ (B) + (-1, 1)$);
  \coordinate (B_2) at ($ (B) + (0, -1.4)$);
  
  \coordinate (D_1) at ($ (D) + (1, -1)$);
  \coordinate (D_2) at ($ (D) + (-1, -1)$);
  \coordinate (E_1) at ($ (E) + (1, 1)$);
  \coordinate (E_2) at ($ (E) + (1, -1)$); 
  \coordinate (F_1) at ($ (F) + (0, 1.4)$);

  \draw[-] (A) -- (B);
  \draw[-] (B) -- (C);
  \draw[-] (C) -- (A);
  
  \draw[-] (D) -- (E);
  \draw[-] (D) .. controls (12, 1) .. (E);  
  \draw[-] (D) .. controls (11, 2) .. (E);  
  \draw[-] (D) -- (F);

  \draw[-] (A) -- (A_1);
  \draw[densely dotted] (A_1) -- ($1.3*(A_1) - 0.3*(A)$);
  \draw[-] (B) -- (B_1);
  \draw[densely dotted] (B_1) -- ($1.3*(B_1) - 0.3*(B)$);
  \draw[-] (B) -- (B_2);
  \draw[densely dotted] (B_2) -- ($1.3*(B_2) - 0.3*(B)$);

  \draw[-] (D) -- (D_1);
  \draw[densely dotted] (D_1) -- ($1.3*(D_1) - 0.3*(D)$);
  \draw[-] (D) -- (D_2);
  \draw[densely dotted] (D_2) -- ($1.3*(D_2) - 0.3*(D)$);
  \draw[-] (E) -- (E_1);
  \draw[densely dotted] (E_1) -- ($1.3*(E_1) - 0.3*(E)$);
  \draw[-] (E) -- (E_2);  
  \draw[densely dotted] (E_2) -- ($1.3*(E_2) - 0.3*(E)$);
  \draw[-] (F) -- (F_1);
  \draw[densely dotted] (F_1) -- ($1.3*(F_1) - 0.3*(F)$);
  
  \draw[red] (C) -- (D);
  \draw[red] (C) -- (F);
  \draw[red] (A) .. controls ($0.5*(A)+0.5*(F) + (0,0.5)$) .. (F);  
  \draw[red] (A) .. controls ($0.5*(A)+0.5*(F) - (0,0.5)$) .. (F);  
  
  \tikzstyle{every node}=[inner sep = 1pt];
  
  \node (dA) at (-0.5,5.5) {$\square^1$};
  \draw[->] (A) .. controls (1,5.5) .. (dA);
  
  \node (dB) at (-1.5,-2) {$\square^2$};
  \draw[->] (B) .. controls (-0.5,-0.5).. (dB);
  
  \node (dC) at (4,-2) {$\square^3$};
  \draw[->] (C) .. controls (4,-0.5).. (dC);
  
  \node (dD) at (13,-2) {$\square^4$};
  \draw[->] (D) .. controls (11.5,-0.5).. (dD);
  
  \node (dE) at (13,1) {$\square^5$};
  \draw[->] (E) -- (dE);
  
  \node (dF) at (12.5,5.5) {$\square^6$};
  \draw[->] (F) .. controls (11,5.5) .. (dF);
  
  	  \draw[densely dotted] (A) ellipse (1 and 1);
	  \draw[densely dotted] (B) ellipse (1 and 1);
	  \draw[densely dotted] (C) ellipse (1 and 1);
	  \draw[densely dotted] (D) ellipse (1 and 1);
	  \draw[densely dotted] (E) ellipse (1 and 1);
	  \draw[densely dotted] (F) ellipse (1 and 1);
\end{tikzpicture}

%% file: G_GC_constr_3.tex
\begin{tikzpicture}[scale=0.35]
  \tikzstyle{every node}=[draw,shape=circle, fill, style={transform shape}];

  \node (A) at (2, 5) { };
  \node (B) at (0, 0) { };
  \node (C) at (4, 0) { };
  \node (D) at (10, 0) { };
  \node (E) at (13, 3) { };
  \node (F) at (10, 5) { }; 
  
  \tikzstyle{every node}=[];

  \coordinate (A_1) at ($ (A) + (0, 1.4)$);
  \coordinate (B_1) at ($ (B) + (-1, 1)$);
  \coordinate (B_2) at ($ (B) + (0, -1.4)$);
  
  \coordinate (D_1) at ($ (D) + (1, -1)$);
  \coordinate (D_2) at ($ (D) + (-1, -1)$);
  \coordinate (E_1) at ($ (E) + (1, 1)$);
  \coordinate (E_2) at ($ (E) + (1, -1)$); 
  \coordinate (F_1) at ($ (F) + (0, 1.4)$);

  \draw[-] (A) -- (B);
  \draw[-] (B) -- (C);
  \draw[-] (C) -- (A);
  
  \draw[-] (D) -- (E);
  \draw[-] (D) .. controls (12, 1) .. (E);  
  \draw[-] (D) .. controls (11, 2) .. (E);  
  \draw[-] (D) -- (F);

  \draw[-] (A) -- (A_1);
  \draw[densely dotted] (A_1) -- ($1.3*(A_1) - 0.3*(A)$);
  \draw[-] (B) -- (B_1);
  \draw[densely dotted] (B_1) -- ($1.3*(B_1) - 0.3*(B)$);
  \draw[-] (B) -- (B_2);
  \draw[densely dotted] (B_2) -- ($1.3*(B_2) - 0.3*(B)$);

  \draw[-] (D) -- (D_1);
  \draw[densely dotted] (D_1) -- ($1.3*(D_1) - 0.3*(D)$);
  \draw[-] (D) -- (D_2);
  \draw[densely dotted] (D_2) -- ($1.3*(D_2) - 0.3*(D)$);
  \draw[-] (E) -- (E_1);
  \draw[densely dotted] (E_1) -- ($1.3*(E_1) - 0.3*(E)$);
  \draw[-] (E) -- (E_2);  
  \draw[densely dotted] (E_2) -- ($1.3*(E_2) - 0.3*(E)$);
  \draw[-] (F) -- (F_1);
  \draw[densely dotted] (F_1) -- ($1.3*(F_1) - 0.3*(F)$);
  
  \draw[-] (C) -- (D);
  \draw[-] (C) -- (F);
  \draw[-] (A) .. controls ($0.5*(A)+0.5*(F) + (0,0.5)$) .. (F);  
  \draw[-] (A) .. controls ($0.5*(A)+0.5*(F) - (0,0.5)$) .. (F);  
  
  \tikzstyle{every node}=[inner sep = 1pt];
  
  \node[red] (dA) at (6,8.5) {$\Delta$};
  \draw[red, ->] (A) .. controls (3,6) .. (dA);
  \draw[red, ->] (B) .. controls (2,3).. (dA);
  \draw[red, ->] (C) .. controls (4.5,3).. (dA);
  \draw[red, ->] (D) .. controls (8.5,2).. (dA);
  \draw[red, ->] (E)  .. controls (9.5,4.5).. (dA);
  \draw[red, ->] (F) .. controls (9,6) .. (dA);
  
  	  \draw[densely dotted] (A) ellipse (1 and 1);
	  \draw[densely dotted] (B) ellipse (1 and 1);
	  \draw[densely dotted] (C) ellipse (1 and 1);
	  \draw[densely dotted] (D) ellipse (1 and 1);
	  \draw[densely dotted] (E) ellipse (1 and 1);
	  \draw[densely dotted] (F) ellipse (1 and 1);
	  
	  \draw[draw=none]  (6.25,2) ellipse (9 and 5.5);
\end{tikzpicture}

%% file: G_GC_constr_4.tex
\begin{tikzpicture}[scale=0.35]
  \tikzstyle{every node}=[draw,shape=circle, fill, style={transform shape}];

  \node (A) at (2, 5) { };
  \node (B) at (0, 0) { };
  \node (C) at (4, 0) { };
  \node (D) at (10, 0) { };
  \node (E) at (13, 3) { };
  \node (F) at (10, 5) { }; 
  
  \tikzstyle{every node}=[];

  \coordinate (A_1) at ($ (A) + (0, 1.4)$);
  \coordinate (B_1) at ($ (B) + (-1, 1)$);
  \coordinate (B_2) at ($ (B) + (0, -1.4)$);
  
  \coordinate (D_1) at ($ (D) + (1, -1)$);
  \coordinate (D_2) at ($ (D) + (-1, -1)$);
  \coordinate (E_1) at ($ (E) + (1, 1)$);
  \coordinate (E_2) at ($ (E) + (1, -1)$); 
  \coordinate (F_1) at ($ (F) + (0, 1.4)$);

  \draw[-] (A) -- (B);
  \draw[-] (B) -- (C);
  \draw[-] (C) -- (A);
  
  \draw[-] (D) -- (E);
  \draw[-] (D) .. controls (12, 1) .. (E);  
  \draw[-] (D) .. controls (11, 2) .. (E);  
  \draw[-] (D) -- (F);

  \draw[-] (A) -- (A_1);
  \draw[densely dotted] (A_1) -- ($1.3*(A_1) - 0.3*(A)$);
  \draw[-] (B) -- (B_1);
  \draw[densely dotted] (B_1) -- ($1.3*(B_1) - 0.3*(B)$);
  \draw[-] (B) -- (B_2);
  \draw[densely dotted] (B_2) -- ($1.3*(B_2) - 0.3*(B)$);

  \draw[-] (D) -- (D_1);
  \draw[densely dotted] (D_1) -- ($1.3*(D_1) - 0.3*(D)$);
  \draw[-] (D) -- (D_2);
  \draw[densely dotted] (D_2) -- ($1.3*(D_2) - 0.3*(D)$);
  \draw[-] (E) -- (E_1);
  \draw[densely dotted] (E_1) -- ($1.3*(E_1) - 0.3*(E)$);
  \draw[-] (E) -- (E_2);  
  \draw[densely dotted] (E_2) -- ($1.3*(E_2) - 0.3*(E)$);
  \draw[-] (F) -- (F_1);
  \draw[densely dotted] (F_1) -- ($1.3*(F_1) - 0.3*(F)$);
  
  \draw[-] (C) -- (D);
  \draw[-] (C) -- (F);
  \draw[-] (A) .. controls ($0.5*(A)+0.5*(F) + (0,0.5)$) .. (F);  
  \draw[-] (A) .. controls ($0.5*(A)+0.5*(F) - (0,0.5)$) .. (F);  
  
  \tikzstyle{every node}=[inner sep = 1pt];
  
  \node[red] (dA) at (6,8.5) {$\square/\Delta$};
  \draw[red, ->] (A) .. controls (3,6) .. (dA);
  \draw[red, ->] (B) .. controls (2,3).. (dA);
  \draw[red, ->] (C) .. controls (4.5,3).. (dA);
  \draw[red, ->] (D) .. controls (8.5,2).. (dA);
  \draw[red, ->] (E)  .. controls (9.5,4.5).. (dA);
  \draw[red, ->] (F) .. controls (9,6) .. (dA);

  \draw[red, densely dotted]  (6.25,2) ellipse (9 and 5.5);
\end{tikzpicture}

%% file: metric2_acknowledgements.tex
The main parts of this paper were developed during the author's Ph.D.\ thesis~\cite{Werner16} supervised by Prof.~J\"urgen~Potthoff, whose constant support the 
author gratefully acknowledges.